\documentclass[reqno]{amsart}

\usepackage[margin=2cm]{geometry}
\usepackage{xypic}
\input xy
\xyoption{all}
\usepackage{epsfig}
\usepackage{color}
\usepackage{amsthm}
\usepackage{amssymb}
\usepackage{amsmath}
\usepackage{amscd}
\usepackage{amsopn}
\usepackage{enumerate}
\usepackage{url}
\usepackage[retainorgcmds]{IEEEtrantools}
\usepackage{datetime}

\usepackage{hyperref}
\hypersetup{colorlinks}

\usepackage{color} %\textcolor{red}{Test}

\definecolor{darkred}{rgb}{1,0,0} %can change the intensity in [0,1]
\definecolor{darkgreen}{rgb}{0,0.8,0}
\definecolor{darkblue}{rgb}{0,0,1}

\hypersetup{colorlinks,
linkcolor=darkblue,
filecolor=darkgreen,
urlcolor=darkred,
citecolor=darkgreen}

\numberwithin{equation}{section}
\newtheorem {Theorem}{Theorem}
\numberwithin{Theorem}{section}

\newtheorem {lemma}[Theorem]    {Lemma}

\newtheorem {proposition}[Theorem]{Proposition}

\theoremstyle{definition}
\newtheorem{Definition}[Theorem]{Definition}
\newtheorem{definition}[Theorem]{Definition}
\theoremstyle{remark}

\newtheorem{remark}[Theorem]{Remark}
\newtheorem{Example}[Theorem]{Example}
\newtheorem{example}[Theorem]{Example}

\newtheorem{notation}[Theorem]{Notation}
\newtheorem{fact}[Theorem]{Fact}

\newcommand{\cF}{{\mathcal F}}
\newcommand{\Fr}{{\mathcal Fr}}
\newcommand{\ve}{{\mathsf{v}}}

\newcommand{\Cc}{{\mathcal C}}
 \newcommand{\Dd}{{\mathcal D}}
 \newcommand{\Ll}{{\mathcal L}}

\newcommand{\Ff}{{\mathcal F}}

\newcommand{\Hh}{{\mathcal H}}

\newcommand{\ip}[2]{\langle #1, #2 \rangle}

\newcommand{\conj}[1]{\overline{#1}}
\newcommand{\inv}{^{-1}}

\DeclareMathOperator{\End}{End}
\DeclareMathOperator{\Hom}{Hom}

\newcommand{\R}{\mathbb{R}}
\newcommand{\C}{\mathbb{C}}
\newcommand{\Z}{\mathbb{Z}}
\newcommand{\LL}{\mathbb{L}}

\newcommand{\inp}[2]{\langle\langle #1, #2 \rangle\rangle}
\newcommand{\vectb}{\overset{\pi}\rightarrow}

\newcommand{\ipd}{\ip{\cdot}{\cdot}}

\newcommand{\Vect}{\mathsf{Vect}}
\newcommand{\Set} {\mathsf{Set}}

\newcommand{\Man} {\mathsf{Man}}
\newcommand{\CoChain} {\mathsf{CoChain}}
\newcommand{\A}{{\mathsf A}}
\newcommand{\B}{{\mathsf B}}
\newcommand{\M}{{\mathsf{Mat}}}
\newcommand{\Finset} {\mathsf{FinSet}}
\newcommand{\FinVect} {\mathsf{FinVect}}

\title{Geometric quantization; a crash course}
\author{Eugene Lerman}
\begin{document}

\setcounter{tocdepth}{1}

\maketitle

\noindent
%{typeset  \timestamp} %\today}

\tableofcontents

%\section*{Forward}
Early in 2011 Sam Evens acting on behalf of the organizers of the
summer school on quantization at Notre Dame asked me to give a short
series of lectures on geometric quantization.  These lectures were
meant to prepare a group of graduate mathematics students for talks at
the conference on quantization which were to follow the summer school.
I was told to assume that the students had attended an introductory
course on manifolds.  But I was not to assume any prior knowledge of
symplectic geometry.  The notes that follow resulted from this
request.  They are a mostly faithful record of four one-hour lectures
(except lecture 4).  It should be said that there exist many books on
geometric quantization starting with Souriau \cite{Souriau}, Sniatycki
\cite{Sniatycki}, Simms and Woodhouse \cite{SimmsWoodhouse} ,
Guillemin and Sternberg \cite{GS_Assymptotics}, Wallach \cite{Wallach}
and Woodhouse \cite{Woodhouse} and continuing with Bates and Weinstein
\cite{BatesWeinstein} and Ginzburg, Guillemin and Karshon \cite{GGK}.
There are also a number of one hundred page surveys on geometric
quantization such as the ones by Ali and Englis \cite{AE} and by
Echeverria-Enriquez {\em et al.}  \cite{EMRV}. Clearly I could not
have squeezed a semester or more worth of mathematics into four
lectures.  Since I had to pick and choose, I decided to convey the
flavor of the subject by proceeding as follows.  In the first lecture
I tried to explain how to formulate the Newton's law of motion in
terms of symplectic geometry.  This naturally require an introduction
of the notions of symplectic manifolds, Hamiltonian vector fields and
Poisson brackets.  In lecture 2 I described prequantization.  Lecture
3 dealt with polarizations.  I have mostly limited myself to real
polarizations. In the original version I tried to explain half-forms.
Here I stick with densities.  In lecture 4 I came back to
prequantization and tried to explain why it is more natural to
prequantize a differential cocycle rather than just a two form.  In
other words prequantization is taken up from a more functorial point
of view --- differential cohomology and stacks.  For reasons of space
and time the treatment is not very detailed.  The notes also contain two
appendices: the first one recalls bits and pieces of category theory;
the second discusses densities. \\

\noindent {\bf Acknowledgments:}\quad I thank the organizers of the summer
school at Notre Dame for inviting me to give the talks and providing
me with a note taker.  I thank the referee for a careful reading of
the manuscript and a number of helpful suggestions.

\section{An outline of the notes}

The goal of this mini course is --- starting with a classical system
(which is modeled as a symplectic manifold together with a function
called the Hamiltonian) --- to produce a quantum system,
that is,  a collection of (skew)adjoined operators on a Hilbert space.\\

Here is a more detailed plan of the lectures (the possibly unfamiliar
terms are to be defined later in the course):
\begin{itemize}
\item We go from Newton's law of motion to a symplectic formulation of
  classical mechanics, while cutting quite a few corners along the
  way.
\item Next we have a crash course on symplectic geometry.  The two key
  points are:
\begin{itemize}
\item A function $h$ on a symplectic manifold $(M, \omega)$ uniquely
  defines a vector field $\Xi_h$ on the manifold $M$.

\item There is a Poisson bracket, that is, an $\R$-bilinear map
\[
C^\infty (M)\times C^\infty (M) \to C^\infty (M), \quad (f,g)\mapsto \{f,g\},
\]
which has a number of properties.  In particular, the bracket $\{f,g\}$ makes
$C^\infty (M)$ into a Lie algebra.

\end{itemize}
\item Next we'll discuss {\sf prequantization}:
Given a symplectic manifold $(M,\omega)$ and a corresponding Poisson
bracket $\{ \cdot, \cdot \}$ we want to find/construct a complex line bundle
$\pi: \LL\rightarrow M$ with a Hermitian inner product
\newline$<\cdot,\cdot>$ and a connection $\nabla$ on $L$ such that:
\begin{itemize}
\item the connection $\nabla$ preserves the inner product $<\cdot,\cdot>$

\item $curv(\nabla)=(2\pi\sqrt{-1})\omega$
\end{itemize}
Given such a bundle we get a prequantum Hilbert Space $\Hh_0$, which
consists of $L^2$ sections of $\LL\rightarrow M$.  We'll observe:
\begin{itemize}
\item Each function $f\in C^{\infty}(M)$ defines an operator
  $Q_f:\Hh_0\rightarrow\Hh_0$,
$$Q_f(s)=(2\pi\sqrt{-1})fs+\nabla_{\Xi_f} s$$
where $\Xi_f$ is the Hamiltonian vector field generated by the function $f$.

\item The map
\[
C^{\infty}(M)\overset{Q}\longrightarrow\{\mathrm{skew\ Hermitian\
  Operators\ on\ } \Hh_0\}
\]
 given by $f\mapsto Q_f$ is a map of Lie algebras.
\end{itemize}
\item There is problem with prequantization: quantum mechanics tells
  us that the space $\Hh_0$ is too big.  Here is an example.
  \begin{Example} Consider a particle in $\R^3$.  The corresponding
    classical phase space is $M=T^*\R^3$.  The associated line bundle
    $\LL=T^*\R^3\times\C$ is trivial, and the prequantum Hilbert space is
    $\Hh_0=L^2(T^*\R^3,\C)$.  Physics tells us that what we
    should have as our quantum phase space the vector space $L^2(\R^3,\C)$.
 \end{Example}
\noindent
 One then uses {\sf polarizations} to cut the number of variables in
 half.  The introduction of polarizations leads to a number of
 technical problems.  In other words, this is where our trouble really begins.
\end{itemize}
In the first appendix to the paper we remind the reader what
categories, functors, natural transformations and equivalences of
categories are.  In the second appendix we discuss densities.

\section{From Newton's law of motion to geometric mechanics in one hour}

%\subsection*{From Newton's Laws to Symplectic geometry}

Consider a single particle  of mass $m$ moving on a line $\R$ subject to a force $F(q,t)$.  Newton's law of motion
in this case says: the trajectory $q(t)$ of the particles solves the second order ordinary differential equation (ODE):
\begin{equation}\label{eq1}
m\dfrac{d^2q}{dt^2}=F(q(t),q'(t),t).
\end{equation}
We now make two simplifying assumptions: (1) the force  $F$ only depends on position $q$  and (2) the force $F$ is
conservative --- that is, $F(q)=-V'(q)$ for some $V\in C^{\infty}(M)$. Then \eqref{eq1} becomes:
\begin{equation}\label{eq2}
m\dfrac{d^2q}{dt^2}=-V'(q(t)).
\end{equation}
The standard way to deal with equation \eqref{eq2} is to introduce a new variable $p$ ("momentum") so that
$p=m\dfrac{dq}{dt}$ and convert \eqref{eq2} into a system of first order ODEs.  That is, if $p=m\dfrac{dq}{dt}$ then
$m\dfrac{d^2q}{dt^2}=\dfrac{dp}{dt}$.  Thus every solution of
\begin{equation}\label{eq.star}
\begin{cases}
\dfrac{dq}{dt}=\dfrac{1}{m}p \\
\dfrac{dp}{dt}=-V'(q)
\end{cases}
\end{equation}
solves \eqref{eq2}.  On the other hand,
a solution of \eqref{eq.star} is an integral curve of a vector field $\Xi(p,q)$, where
\[
 \Xi(p,q)=\dfrac{1}{m}p\ \frac{\partial}{\partial q}-V'(q)\ \frac{\partial}{\partial p}.
\]
 Note: the energy $H(q,p)=\dfrac{1}{2m}p^2+V(q)$ is conserved, i.e., it is constant along solutions of \eqref{eq.star}. In fact the function $H$ completely determines the vector field $\Xi$ in the following sense.
 Consider the two-form $\omega=dp\wedge dq$ on $\R^2$. It is easy to see that
 \[
 \iota(\Xi)\omega=-dH.
 \]
 So H determines $\Xi$.
Notice $\omega=dp\wedge dq$ is nondegenerate, so for any 1-form $\alpha$ the equation $\omega(X,\cdot)=\alpha(\cdot) $ has a unique solution.
To summarize: Newton's equations and $m\dfrac{d^2q}{dt}=-V'(q)$ are equivalent to integrating the vector field $\Xi_H$ defined by $\iota(\Xi_H)\omega =-dH$. We now generalize this observation.

\begin{Definition} A {\sf symplectic form} $\omega$ on a manifold $M$
  is a closed nondegenerate 2-form.  The pair $(M, \omega)$ is called
  a {\sf symplectic manifold}.
\end{Definition}
\begin{remark}
  $d\omega=0$ will give us an important property of the Poisson
  bracket: the Jacobi identity.
\end{remark}
\subsection*{Standard examples of symplectic manifolds}
\begin{example}
 $ (\R^2,dp\wedge dq)$
\end{example}

\begin{example}$(\Sigma,\omega)$ where $\Sigma$ is an orientable
  surface and $\omega$ is an area form (nowhere zero form) on
  $\Sigma$.  Note that since the surface $\Sigma$ is two dimensional,
  $d\omega$ is automatically 0.
\end{example}

\begin{example}\label{ex:2.5} 
  Let $Q$ be any manifold and set $M=T^*Q$. If $\alpha$ denotes the
  tautological 1-form then $\omega=d\alpha$ is a symplectic form on
  $M$.  Here are some details. ``Recall" that there are two ways of
  defining the tautological 1-form (also called the Liouville form)
  $\alpha$.
     \begin{enumerate}
     \item In local coordinates $\alpha$ is defined as follows.  If
       $(q_1, \ldots q_n)$ is a coordinate chart on $Q$ and $(q_1,
       \ldots, q_n, p_1, \ldots, p_n)$ the corresponding coordinates
       on the cotangent bundle $T^*M$, then
    \[
    \alpha=\sum p_idq_i.
    \]
    It is not hard to check that $d\alpha=\omega=\sum dp_i\wedge dq_i$
    is non-degenerate. It is closed by construction, hence it's
    symplectic.  It is not obvious that $\alpha$ (and hence $\omega$)
    are globally defined forms.

\item Alternatively, we have the projection
  $\pi:T^*Q\rightarrow Q$ and $d\pi:T_{(q,p)}(T^*Q)\rightarrow T_qQ$.
  So given $q\in Q$, $p\in T_q^*Q$ and $v\in T_{(q,p)}(T^*Q)$ define
\[
    \alpha_{(q,p)}(v)=p(d\pi(v)).
\]
\end{enumerate}
It is a standard exercise to check that the two constructions agree.
In the first construction of $\alpha$ it is clear  that $\alpha$ is
smooth and $d\alpha$ is nondegenerate. In the second construction it
is clear that $\alpha$ is globally defined.
\end{example}

\begin{Definition} The {\sf Hamiltonian vector field} $\Xi_f$ of a
  function $f$ on a symplectic manifold $(M,\omega)$ is the unique
  vector field defined by $\omega(\Xi_f,\cdot)=-df$.
\end{Definition}
\noindent
{\bf Warning}: the opposite sign convention is also frequently used in
literature: $\omega(\Xi_f,\cdot)=df$.
\begin{remark}
The function  $f$ is always constant along the integral curves of its Hamiltonian vector field $\Xi_f$.
\end{remark}
\begin{proof}
Let $\gamma$ be an integral curve of the vector field $\Xi_f$.  Then  \[
\dfrac{d}{dt}f(\gamma(t))=\Xi_f(f)=df(\Xi_f)=-\omega(\Xi_f,\cdot)(\Xi_f)=-\omega(\Xi_f,\Xi_f).
\]
Since $\omega$ is skew-symmetric, $\omega(\Xi_f,\Xi_f)=0$.  Hence $\dfrac{d}{dt}f(\gamma(t))=0$, i.e., $f(\gamma(t))$ is a constant function of $t$,
which is what we wanted to prove.
\end{proof}
\subsection*{Poisson Bracket}
\begin{Definition}
The {\sf Poisson bracket} $\{\cdot,\cdot\}$ on a symplectic manifold $(M,\omega)$ is a map
\[
\{\cdot,\cdot\}:C^\infty(M)\times C^\infty(M)\rightarrow C^\infty(M)
\]
defined by
\[
\{f,g\}=\Xi_f(g).
\]
\end{Definition}

\begin{remark}
The Poisson bracket has a number of useful properties which we list below.
Proofs may be found in any symplectic geometry book.
\begin{enumerate}
\item For any three function $f,g,h\in C^\infty(M)$ we have
\[
\{f,g \cdot h\}=\Xi_f(g\cdot h)=\Xi_f(g)\cdot h+g\cdot\Xi_f(h)=\{f,g\}\cdot h+g\cdot\{f,h\}.
\]
\item  For any pair of functions $f,g$ we have
$\{f,g\}=\Xi_f(g)=dg(\Xi_f)=-\omega(\Xi_g,\Xi_f)=\omega(\Xi_f,\Xi_g)=-\{g,f\}$. In particular $\{f,f\}=0$.

\item One can show that the equation $d\omega=0$ implies that
\[
\{f,\{g,h\}\}=\{\{f,g\},h\}+\{g,\{f,g\}\},
\]
which is the Jacobi identity.  In other words the pair $(C^\infty (M), \{\cdot, \cdot\})$ is a Lie algebra.
\item  It is not hard to show that $\{f,g\}=0$ if and only if $g$ is constant along integral curves of $f$.
Indeed,  let $\gamma$ be an integral curve of the vector field $\Xi_f$.  Then
\[
\dfrac{d}{dt}g(\gamma(t))=\Xi_f(g)\, (\gamma(t))= \{f,g\} (\gamma (t)).
\]
Therefore if the function $g$ is constant along $\gamma$ then the bracket $\{f,g\}$ is zero along $\gamma$.   The converse is true as well.  This generalizes the fact that a function $f$ is constant along the integral curves of its Hamiltonian vector field $\Xi_f$.

\item One can show that $\iota([\Xi_f,\Xi_g])\omega=-d\{f,g\}$.  Hence if the Poisson bracket $\{f,g\}$ of two functions $f,g$ is  $0$ then flows of their Hamiltonian vector fields commute.  Here is a better interpretation of the same fact:
    The map
    \[
    C^\infty (M) \to \textrm{ vector fields on }M,\quad f\mapsto \Xi_f
    \]
    is a map of Lie algebras: $\Xi_{\{f,g\}}=[\Xi_f,\Xi_g]$.

\end{enumerate}
\end{remark}

\begin{remark}
  It is not hard to show that is $(M, \omega)$ is a symplectic
  manifold, then its dimension is necessarily even.  This only
  involves linear algebra.

  Say $\dim M = 2n$.  Then one can show further that the $2n$-form
  $\omega^n:= \overbrace{\omega \wedge \cdots \wedge \omega}^n $
  ($n$-fold wedge product) is nowhere zero, hence defines an
  orientation of $M$.  In particular, this allows us to integrate any
  compactly supported function $f\in C^\infty (M)$ over $M$ by
  integrating the form $f\omega^n$.  The space $L^2 (M, \omega)$ is
  then defined as the completion of the space $C_c^\infty(M, \C)$ of
  compactly supported functions with respect to the $L^2$ norm
\[
\|f\|:= \left(\int _M |f|^2 \omega^n \right)^{1/2}.
\]
It is a Hilbert space with the Hermitian inner product
\[
\langle\langle f,g\rangle\rangle := \int _M \bar{f}g\omega^n
\]
(in the convention I prefer, the Hermitian inner products are
complex-linear in the {\sf second} variable).
\end{remark}

% iven $\omega\in\Omega^2(M)$ with $d\omega=0$ and $\omega$ nondegenerate then:

% 1. $dim(M)=2n$, $M$ is orientable and $\omega^n$ defines an orientation.
% %\end{remark}

% Consequences:

% 1. $\{f,g\}=0$ implies the flows of $\Xi_f$ and $\Xi_f$ commute.

% 2. $(C^\infty(M),\{\cdot,\cdot\})$ is a Lie algebra.

% 3. The map $C^\infty(M)\rightarrow \Gamma(TM)$ defined by $f\mapsto\Xi_f$ is a map of Lie algebras, $\Xi_{\{f,g\}}=[\Xi_f,\Xi_g]$.
\noindent
We end the section with an easy lemma that we will need later.

\begin{lemma}\label{lem1.9}
  Let $(M,\omega)$ be a symplectic manifold and $f:M\to \R$ a smooth
  function. Then the Lie derivative $\Ll_{\Xi_f}\omega$ of the
  symplectic form with respect to the Hamiltonian vector field of the
  function $f$ is zero:
\[
\Ll_{\Xi_f}\omega =0.
\]
\end{lemma}

\begin{proof}
  This is an application of Cartan's formula: for a differential form
  $\sigma$ its Lie derivative $\Ll_X\sigma$ with respect to a vector
  field $X$ is given by $\Ll_X \sigma = \iota(X)d \sigma + d
  \iota(X)\sigma$, where, as above, $\iota(X)\sigma$ denotes the
  contraction of $X$ and $\sigma$, etc.  Thus
\[
\Ll_{\Xi_f}\omega = \iota (\Xi_f)d\omega + d \iota (\Xi_f)\omega.
\]
The first summand above is 0 since $d\omega=0$.  By definition of
$\Xi_f$, $\iota (\Xi_f)\omega = -df$.  Since $d(df) =0$, the second
summand is zero as well.
\end{proof}

\section{Prequantization}

The goal of this section is to describe geometric 
prequantization.  This is a procedure for turning a classical
mechanical system mathematically formalized as a symplectic manifold
$(M, \omega)$ together with its Poisson algebra of smooth functions
$C^\infty(M)$ (``classical observables'') into a quantum system
formalized as a Hilbert space $\Hh_0$ together with the Lie algebra
of (densely defined) skew-Hermitian operators $\{Q_f| f\in C^\infty
(M)\}$.  Moreover the map
\[
Q: C^\infty (M) \to \End(\Hh_0), \quad f\mapsto Q_f,
\]
should  (and would) be a map of Lie algebras:
\[
Q_{\{f,g\}} = [Q_f, Q_g]
\]
for all functions $f,g\in C^\infty(M)$.  We start by recalling some
notation.
\begin{notation}
  We denote the space of sections of a vector bundle $E\to M$ over a
  manifold $M$ by $\Gamma (E)$.  Thus the space of vector fields on a
  manifold $M$ is denoted by $\Gamma (TM)$.
\end{notation}

% Recall:

% 2. For all $f\in C^\infty(M)$ there exists a unique $\Xi_f\in \chi(M)\equiv\Gamma(M)$ defined by $\iota(\Xi_f)\omega=-df$.

% 3.$\{\cdot,\cdot\}:C^\infty(M)\times C^\infty(M)\rightarrow C^\infty(M)$ given by
% $$\{f,g\}=\omega(\Xi_f,\Xi_g)=-df(\Xi_g)=dg(\Xi_f)$$ satisfies the Jacobi identity.

% 4. The map $C^\infty(M)\rightarrow \chi(M)$ defined by $f\mapsto\Xi_f$ is a map of Lie algebras.

\subsection{Connections}
We start by recalling a number of standard definitions and facts.  By a
{\sf fact} I mean a theorem the proof of which will take us too far afield.
Such proofs may be found in any number of textbooks.
\begin{definition}
A {\sf connection} $\nabla$ on a (complex) vector bundle $E\overset{\pi}\rightarrow M$ is a $\C$-bilinear map
\[
\Gamma(TM)\times\Gamma(E)\rightarrow\Gamma(E), \quad (X,s)\mapsto
\nabla_Xs
\]
 such that
\begin{enumerate}
\item $\nabla_{fX}s=f\nabla_Xs$ for all functions $f\in C^\infty
  (M,\C)$ and all vector fields $X\in \Gamma (TM)$ (i.e., $\nabla$ is
  $C^\infty(M,\C)$ linear in the first variable) and
\item $\nabla_X(fs)=X(f)s+f\nabla_Xs$ for all functions $f\in C^\infty
  (M,\C)$ and all vector fields $X\in \Gamma (TM)$ (i.e., $\nabla $ is
  a derivation in second slot).
\end{enumerate}

\end{definition}
If the vector bundle $E$ carries a fiber-wise Hermitian inner product
$\ip{\cdot}{\cdot}$ we can talk about the connections respecting this
structure.  More precisely
\begin{definition}
  A connection $\nabla$ on a vector bundle $E\overset{\pi}\rightarrow
  M$ with a fiber-wise inner product $\ip{\cdot}{\cdot}$ is {\sf
    Hermitian} if
\[
X(\ip{s}{s'})=\ip{\nabla_Xs}{s'}+\ip{s}{\nabla_Xs'}
\] for all vector fields $X$ on $M$ and all section $s,s'\in \Gamma (E)$.
\end{definition}

\noindent
Next recall that given any complex vector bundle
$E\overset{\pi}\rightarrow M$ we can consider the bundle
$\End(E)\rightarrow M$ of endomorphisms with a fiber $\End(E)_x$ at a point $x\in M$
given by
\[
\End(E)_x=\{A:E_x\rightarrow E_x \mid A \mathrm{\ is\ } \C\
\mathrm{linear}\}.
\]
We also have the subbundle of
skew Hermitian maps $\End(E,\ip{\cdot}{\cdot})\subset \End(E)$ with typical fiber
\[
\End(E,\ip{\cdot}{\cdot})_x=\{A:E_x\rightarrow E_x \mid A\mathrm{\ is\ }
\C\ \mathrm{linear\ and \ } \ip{Av}{w}+\ip{w}{Av}=0 \ \ \forall v,w\in
E_x\}.
\]

\noindent
We have the following {\em Fact} :
\begin{fact} The space of Hermitian connections on a vector bundle
  $(E,\ip{\cdot}{\cdot})$ is non empty.  In fact it
  is %actually % big and forms
  an infinite dimensional affine space: the difference of two
  connections is a $\End(E,\ip{\cdot}{\cdot} )$ valued 1-form.
\end{fact}
\noindent
\begin{definition}
  Let $\nabla:\Gamma(TM)\times\Gamma(E)\rightarrow\Gamma(E)$ be a
  connection on a vector bundle $E\to M$.  The {\sf curvature}
  $R^\nabla$ of the connection is a section of $\Lambda^2(T^*M)\otimes
  \End(E)$ (i.e., an $\End(E)$ valued 2-form). It is defined by
\[
R^\nabla(X,Y)s=\nabla_X(\nabla_Ys)-\nabla_Y(\nabla_Xs)-\nabla_{[X,Y]}s
\]
for all vector fields $X$ and $Y$ and all sections $s\in \Gamma(E)$.
\end{definition}
If $\nabla$ is a Hermitian connection then its curvature $R^\nabla$
is a 2-form with values in $\End(E,\ip{\cdot}{\cdot})$. Furthermore if
$E\overset{\pi}\rightarrow M$ is a complex line bundle then
\[
\End(E,\ip{\cdot}{\cdot})\simeq M\times\sqrt{-1}\R
\]
hence
\[
\frac{1}{\sqrt{-1}}R^\nabla\in\Omega^2(M,\R).
\]
That is $\frac{1}{\sqrt{-1}}R^\nabla$ is an ordinary real valued 2-form.

Fix a manifold $M$ and consider the collection $\mathcal{D}(M)$ of all
triples $(L, \ip{\cdot}{\cdot}, \nabla)$, where $L\to M$ is a complex line bundle, $\ip{\cdot}{\cdot}$ is a Hermitian inner product on $L$ and $\nabla$ is a Hermitian connection on $(L, \ip{\cdot}{\cdot})$.  Then curvature defines a map
\begin{equation}\label{eq:2.1}
\mathcal{D}(M)\to \Omega^2(M), \quad (L, \ip{\cdot}{\cdot}, \nabla) \mapsto
\frac{1}{\sqrt{-1}}R^\nabla
\end{equation}
from the collection $\mathcal{D}(M)$ to the set of (real-valued)
2-forms $\Omega^2(M)$.  

To define {\sf geometric prequantization} one needs to invert this map.
That is, given a symplectic manifold $(M,\omega)$ one would like to
find a Hermitian line bundle with  a
Hermitian connection $\nabla$ so that
\[
\frac{1}{\sqrt{-1}}R^\nabla =
\omega.
\]
However there are two problems: (1) the map \eqref{eq:2.1} is
not 1-1 and (2) not all symplectic forms are in the image of the map.
The first problem has to do with the fact that taking curvature of a
connection is very much like taking the exterior derivative of a
1-form.  So recovering connection from its curvature is also like
recovering a 1-form from its exterior derivative.

The second problem is topological.  It has to do with the fact that
(isomorphism classes of) complex line bundles are parametrized by
degree 2 integral cohomology classes, that is, elements of $H^2 (M,
\Z)$.  Moreover the cohomology class $c_1(M)$ of a line bundle $E\to
M$ and the de Rham class $[\frac{1}{2\pi\sqrt{-1}}R^\nabla]$ defined
by the curvature $R^\nabla$ of a Hermitian connection $\nabla$ on $E$
are closely related: $[\frac{1}{2\pi\sqrt{-1}}R\nabla]$ is the image
of $c_1(E)$ under the natural map
\[
\iota: H^2(M,\Z) \to H^2 (M, \Z)\otimes \R \simeq H^2_{dR}(M).
\]
Consequently the integral of the 2-form
$\frac{1}{2\pi\sqrt{-1}}R^\nabla$ over any smooth integral 2-cycle in
$M$ has to be an integer.  Hence the only symplectic forms that
can be prequantized (that is, can occur as curvatures) are the forms
whose integration over integral 2-cycles give integers.  And if a
symplectic form $\omega$ is integral (that is, the de Rham class
$[\omega]$ lies in the image of the map $\iota$ above), the lift of
$[\omega]$ to $H^2 (M, \Z)$ need not be unique.  A
solution to these problems (independently due to Kostant and to
Souriau) can be stated as follows:

\begin{Theorem}\label{thm:3.6} Suppose the de Rham cohomology class of a closed
  2-form $\sigma$ on a manifold $M$ lies in the image of $\iota:
  H^2(M,\Z) \to H^2_{dR} (M)$. Then there exists a Hermitian line
  bundle $E\to M$ with a Hermitian connection $\nabla$ such that
  $\frac{1}{2\pi\sqrt{-1}}R^\nabla=\sigma$.
\end{Theorem}
There is another solution to this problem that I find more
satisfactory and to which Theorem~\ref{thm:3.6}  is a corollary. It has the
additional merit of allowing one to prequantize orbifolds as well.
It involves thinking of $\mathcal{D}(M)$ not just as a set but as a
collection of objects in a category and upgrading the map
\eqref{eq:2.1} to a functor. The target of this functor is a category
of differential cocycles: the objects of this category involve integral
cocycles {\em and } differential forms.  The functor will turn out to
be an equivalence of categories.  So it can be easily inverted (up to
homotopy). We will take this up in the last section of the notes.  In
the mean time we proceed with prequantization.

\begin{definition}
Suppose $(E\to M,\ip{\cdot}{\cdot},\nabla)$ is a Hermitian line
bundle with connection such that
$\omega:=\frac{1}{2\pi\sqrt{-1}}R^\nabla$ is symplectic.
The {\sf prequantization} is a linear map
\[
 Q:C^\infty(M)\rightarrow \Hom(\Gamma(E),\Gamma(E)), \quad f\mapsto Q_f
\]
where the operator $Q_f$ is defined by
\[
Q_f(s)=\nabla_{\Xi_f}s-2\pi\sqrt{-1}f\cdot s
\]
for all functions $f\in C^\infty (M)$ and all section $s\in \Gamma
(E)$.  Here, as before, $\Xi_f$ denotes the Hamiltonian vector field
of $f$ with respect to the symplectic form $\omega$.
\end{definition}

\begin{remark}
  Our definition of $Q$ differs from a more traditional one by
  $\sqrt{-1}$.  The physicists like to identify the Lie algebra of the
  unitary group with Hermitian matrices (and operators).
\end{remark}

We next prove:
\begin{lemma}
  The prequantization map $Q:C^\infty(M)\rightarrow
  \Hom(\Gamma(E),\Gamma(E))$ is a map of Lie algebras:
\[
[Q_f,Q_g]s=Q_{\{f,g\}}s
\]
for all sections $s\in \Gamma (E)$ and all functions $f,g\in C^\infty
(M)$.
\end{lemma}

\begin{proof}
  Since $\omega:=\frac{1}{2\pi\sqrt{-1}}R^\nabla$, we have, by
  definition of curvature that
\[
([\nabla_X,\nabla_Y]-\nabla_{[X,Y]})s=2\pi\sqrt{-1}\omega(X,Y)\cdot s
\]
for all vector fields $X,Y$ on $M$ and all sections $s\in \Gamma (E)$.
Hence
\[
[\nabla_{\Xi_f},\nabla_{\Xi_g}] s = \nabla_{[\Xi_f,\Xi_g]}s +
2\pi\sqrt{-1}\omega(\Xi_f , \Xi_g)\cdot s
\]
for all $f,g\in C^\infty (M)$. Since $\omega(\Xi_f , \Xi_g)=
{\{f,g\}}$ and since $[\Xi_f,\Xi_g] =\Xi _{\{f,g\}}$ we get
\begin{equation}\label{eq:2.2}
[\nabla_{\Xi_f},\nabla_{\Xi_g}] s =
\nabla_{\Xi _{\{f,g\}}}s  +2\pi\sqrt{-1}{\{f,g\}}s.
\end{equation}
Next observe that
\begin{IEEEeqnarray*}{rCl}
  Q_f(Q_gs)&=&Q_f(\nabla_{\Xi_g}s-2\pi\sqrt{-1}g\cdot s)\\
  &=&\nabla_{\Xi_f}(\nabla_{\Xi_g}s-2\pi\sqrt{-1}g\cdot
  s)-2\pi\sqrt{-1}f\cdot(\nabla_{\Xi_g}s-2\pi\sqrt{-1}g\cdot s)\\
  &=&\nabla_{\Xi_f}(\nabla_{\Xi_g}s)-2\pi\sqrt{-1}\Xi_f(g)\cdot
  s-2\pi\sqrt{-1}g\nabla_{\Xi_f}s-2\pi\sqrt{-1}f\nabla_{\Xi_g}s-4\pi^2fg\cdot s
\end{IEEEeqnarray*}
Similarly,
\[
Q_g(Q_fs)=\nabla_{\Xi_g}(\nabla_{\Xi_f}s)-2\pi\sqrt{-1}\Xi_g(f)\cdot s-2\pi\sqrt{-1}f\nabla_{\Xi_g}s-2\pi\sqrt{-1}g\nabla_{\Xi_f}s-4\pi^2g\cdot f\cdot s.
\]
Hence
\begin{IEEEeqnarray*}{rCl}
[Q_f,Q_g]s&=& Q_f(Q_gs)-Q_g(Q_fs)\\
&=&[\nabla_{\Xi_f},\nabla_{\Xi_g}]s-2\pi\sqrt{-1}(\{f,g\}-\{g,f\})s\\
&=&[\nabla_{\Xi_f},\nabla_{\Xi_g}]s-4\pi\sqrt{-1}\{f,g\}s\\
&=& \nabla_{\Xi _{\{f,g\}}}s  +2\pi\sqrt{-1}{\{f,g\}}s
-4\pi\sqrt{-1}\{f,g\}s\quad \textrm{ by \eqref{eq:2.2}}\\
&=& \nabla_{\Xi _{\{f,g\}}}s-2\pi\sqrt{-1}{\{f,g\}}s = Q_{\{f,g\}}s.
\end{IEEEeqnarray*}

\end{proof}

\begin{definition}\label{def.square-integrable}
  Let $(M,\omega)$ be a symplectic manifold of dimension $2m$ and
  $(E\to M, \ip{\cdot, \cdot})$ be a Hermitian line bundle as before.
  A section $s\in \Gamma (E)$ is {\sf square integrable} if the
  integral $\int _M \langle s,s\rangle\, \,\omega^m $ converges.  
\end{definition}

Clearly any section $s$ with compact support is square integrable.  
Moreover, for any two compactly supported sections $s,s'$ of $E\to M$ the function $\langle s, s'\rangle$ is compactly supported, hence the integral
$\int_M\ip{s}{s'}\, \,\omega^m$ converges.

\begin{notation}
We denote the space of compactly supported sections of the bundle $E\to M$ by $\Gamma_c (E)$:
\[
\Gamma_c (E):= \{s\in \Gamma(E)\mid \mathrm{supp} (s) \,\textrm{
  is compact }\}.
\]
\end{notation}

\noindent
The space $\Gamma_c (E)$ of compactly supported sections carries a
natural Hermitian inner product defined by
\[
\inp{s}{s'}=\int_M\ip{s}{s'}\, \,\omega^m
\]
for all  $s,s'\in \Gamma_c(E)$.

\begin{definition} The {\sf prequantum Hilbert space} associated with
  a prequantum line bundle $(E\to M, \ip{\cdot, \cdot})$ is the
  completion of the inner product space
  $(\Gamma_c(E),\inp{\cdot}{\cdot})$ with the respect to the
  corresponding $L^2$ norm:
\[
\Hh_0:= \textrm{ the completion of } \Gamma_c (E).
\]
\end{definition}

\begin{lemma}
  The prequantization map $Q:C^\infty(M)\to \Hom(\Gamma_c(E),\Gamma_c(E))$
  is skew-Hermitian:
\[
\inp{Q_fs}{s'}+\inp{s}{Q_fs'}=0
\]
for all compactly supported sections $s,s'$ of $E\to M$.
\end{lemma}

\begin{proof} Observe that
\[
\ip{2\pi\sqrt{-1}fs}{s'}+\ip{s}{2\pi\sqrt{-1}fs'}=0
\]
for all functions $f\in C^\infty(M,\R)$ and all square-integrable
sections $s,s'$.  Next note that since the connection $\nabla$ is
Hermitian we have
\[
\int_M\ip{\nabla_{\Xi_f}s}{s'}\omega^m+\int_M\ip{s}{\nabla_{\Xi_f}s'}\omega^m
=\int_M\Xi_f\ip{s}{s'}\omega^m.
\]
Since the Lie derivative $\Ll_{\Xi_f}\omega$ of the symplectic form with
respect to any Hamiltonian vector field $\Xi_f$ zero (see Lemma~\ref{lem1.9}),
we have $\Ll_{\Xi_f} \omega^m = 0$ as
well.  Hence
\[
  \Ll_{\Xi_f}(\ip{s}{s'}\omega^m)=\Ll_{\Xi_f}(\ip{s}{s'})\omega^m+\ip{s}{s'}\Ll_{\Xi_f}(\omega^m)=\Xi_f\ip{s}{s'}\omega^m + 0.
\]
On the other hand by Cartan's magic formula, for any top degree form
$\mu$ we have
\[
\Ll_{\Xi_f} \mu = \iota (\Xi_f) d \mu + d (\iota (\Xi_f)\mu) =
d (\iota (\Xi_f)\mu),
\]
since $d \mu =0$.  % By Stokes' theorem
% \[
% \int_M d(\iota(\Xi_f)\ip{s}{s'}\omega^m)=0.
% \]
We conclude that
\begin{IEEEeqnarray*}{rCl}
\int_M\ip{\nabla_{\Xi_f}s}{s'}\omega^m+
\int_M\ip{s}{\nabla_{\Xi_f}s'}\omega^m&=&\int_M\Xi_f\ip{s}{s'}\omega^m\\
&=&\int_M \Ll_{\Xi_f}\left(\ip{s}{s'}\omega^m\right)\\
&=&\int_M d\left(\imath(\Xi_f)\ip{s}{s'}\omega^m\right)=0,\\
\end{IEEEeqnarray*}
where the last equality holds by Stokes' theorem.
The result follows.
\end{proof}

\begin{remark}
  Note that the operators $Q_f$ are not bounded in the $L^2$ norm
  since they involve differentiation.  So they do not extend to
  bounded operators on the completion $\Hh_0$.  However, they are
  elliptic operators, and consequently extend to closed densely defined
  operators on $\Hh_0$.  Not surprisingly their domain of definition
  consists of square integrable sections with square integrable
  (distributional) first derivatives.
\end{remark}

\section{Polarizations}
%%%%%%%%%%%%%%%%%%%%%%%%%%%%%%

Recall that geometric prequantization associates to an integral
symplectic manifold $(M, \omega)$ a Hilbert space $\Hh_0$ and
to each real-valued function $f$ on $M$ a skew-Hermitian operator
$Q_f:\Hh_0\to \Hh_0$.    Unfortunately this is not correct
physics.

\begin{example} Suppose our classical configuration space is $\R$, the
  real line.  This is the example with which we started these
  lectures.  The corresponding classical phase space is the
  cotangent bundle $M=T^*\R$ with the canonical symplectic form
  $\omega=dp\wedge dq$.  The corresponding prequantum line bundle $E$
  is trivial: $E= T^*\R\times\C \to T^*\R$.  Hence the prequantum
  Hilbert space is $\Hh_0$ is the space $L^2(T^*\R,\C)$ of complex
  valued square integrable functions.  Quantum mechanics tells us that
  the correct Hilbert space consists of square-integrable functions of
  one variable $L^2(\R,\C)$, not of functions of two variables.
\end{example}

A standard solution to this problem is to introduce a polarization.  To
define polarizations we start with linear algebra.

\begin{Definition}
  A {\sf Lagrangian subspace} $L$ of a symplectic vector space
  $(V,\omega)$ is a subspace satisfying two conditions:
\begin{enumerate}
\item $L$ is {\sf isotropic}: $\omega(v,v')=0$ for all vectors  $v,v'\in L$;

\item $L$ is {\sf maximally} isotropic: for any isotropic subspace
  $L'$ of $V$ containing $L$ we must have $L= L'$.
\end{enumerate}
\end{Definition}
\begin{remark}
A standard argument shows that if $L\subset (V,\omega)$ is Lagrangian then
\[
\dim L = \frac{1}{2} \dim V.
\]
\end{remark}
\begin{Example} If $(V,\omega)= (\R^2,\omega=dp\wedge\ dq)$ then any
  line $L$ in $V$ is Lagrangian.
\end{Example}

\noindent
The analogous definition for submanifolds is as follows:
\begin{definition}
  An immersed submanifold $L$ of a symplectic manifold $(M,\omega)$ is
  {\sf Lagrangian} if $T_xL \subset (T_xM, \omega_x)$ is a Lagrangian
  subspace for each point $x\in L$.
\end{definition}
\begin{Definition}
  A (real) {\sf polarization} on a symplectic manifold $(M,\omega)$ is
  a subbundle $\Ff\subset TM$ of its tangent bundle such that
\begin{enumerate}
\item $\Ff$ is {\sf Lagrangian}: $\Ff_x\subset (T_xM, \omega_x)$ is a
  Lagrangian subspace for all points $x\in M$.
\item $\Ff$ is {\sf integrable} (or {\sf involutive}): for all local
  sections $X,Y$ of $\Ff\to M$, the Lie bracket $[X,Y]$ is again a
  local section of $\Ff$. This conditions is often abbreviated as
  $[\Ff,\Ff]\subset \Ff$.
\end{enumerate}
\end{Definition}

\begin{remark}
  If $\Ff\subset TM$ is an integrable distribution, then by the Frobenius
  theorem there exists a foliation $\mathcal{L}_\Ff$ of $M$ tangent to
  the distribution $\Ff$.  If in addition $\Ff$ is Lagrangian then the
  leaves of $\mathcal{L}_\Ff$ are immersed Lagrangian submanifolds of
  $M$.
\end{remark}

\begin{example} Suppose the symplectic manifold $ M$ is a cotangent
  bundle of some manifold $N$: $M=T^*N$ with its standard symplectic
  form.  Then $M$ has a polarization $\Ff$ given by the kernel of the
  differential of $\pi:T^*N\to N$:
\[
\Ff=\ker(d\pi):T(T^*N)\rightarrow TN.
\]
In local coordinates $(q_1, \ldots, q_n, p_1, \ldots, p_n):T^*U\to \R^n\times \R^n$ ($U\subset N$ open) on $T^*N$ this polarization is given by
\[
\Ff= \left\{ \sum_{i=1}^n a_i\frac{\partial}{\partial p_i}
 \mid a_i \in C^\infty (U)\right\}.
\]
The polarization $\Ff$ is called the {\sf vertical polarization of
  $T^*N$}.
\end{example}

\begin{example}\label{ex:3.punc}
  Consider the punctured plane $M:= \R^2 \smallsetminus \{0\}$ with
  the symplectic form $dp \wedge dq$.  The collection of circles
  $\left\{ C_r := \{(p,q) \in M \mid p^2 +q^2 = r^2\}\right\}_{r>0}$
  forms a Lagrangian foliation of $M$.  The tangent lines to the
  circle define a polarization $\cF$ of $M$.
\end{example}

\begin{remark} Real polarizations need not exist. Here is an example.
  It is not hard to show that any real line bundle over the two-sphere
  $S^2$ has to be trivial, hence has to have a nowhere zero section.
  Thus if $S^2$ has a polarization then it has a nowhere zero vector
  field, which contradicts a well-known theorem.

  Since real polarization need not exist for interesting classical systems, one generalizes % The problem of existence of real polarizations leads to the
%   generalization of 
  the notion of a real polarization to that of a {\sf complex}
  polarization.  A complex polarization on a symplectic manifold
  $(M,\omega)$ is a complex Lagrangian involutive subbundle of the
  complexified tangent bundle $TM\otimes \C$.  We will not discuss
  them further. There is a well-developed theory of complex
  polarizations that an interested reader may consult.

  Finally there are examples due to Mark Gotay
\cite{GotayNoP} of symplectic manifolds
  that admit no polarizations whatsoever, real or complex.  One can
  reconcile oneself to this fact by thinking that not all classical
  mechanical systems have quantum counterparts.
\end{remark}

\begin{Definition} Let $\Ff$ be a polarization on $(M,\omega)$ and
  $(\LL\vectb M,\ip{\cdot}{\cdot},\nabla)$ a prequantum line bundle. A
  section $s\in\Gamma(\LL)$ is {\sf covariantly constant along $\Ff$} if
  $\nabla_Xs=0$ for all sections $X\in\Gamma(\Ff)$.
\end{Definition}

\begin{notation}
  We denote the space of sections of the prequantum line bundle $\LL\to
  M$ covariantly constant along a polarization $\Ff$ by $\Gamma_\Ff
  (\LL)$.  Thus
\[
\Gamma_\Ff (\LL):= \{s\in \Gamma(\LL) \mid \nabla_Xs=0\textrm{ for all }
X\in \Gamma(\Ff)\}.
\]
\end{notation}
\begin{remark} It is common to refer to the space $\Gamma_\Ff (\LL)$
  as the space of {\sf polarized sections}.
\end{remark}
\begin{example}
  If $M =T^*N$, $\LL= T^*N\times \C$ and $\Ff \subset TM$ is the vertical
  polarization, then the space $\Gamma_\Ff (\LL) \subset C^\infty (T^*N,
  \C)$ consists of functions constant along the fibers of $\pi:T^*N\to
  N$.  Thus  $\Gamma_\Ff (\LL) = \pi^* C^\infty (N, \C)$.
\end{example}

\begin{example} \label{ex:3.punct}
Consider again the punctured plane $M= \R^2
  \smallsetminus \{0\}$ with the polarization $\cF$ defined by
  circles.  We may take the trivial bundle $\LL= M\times \C\to M$ as the
  prequantum line bundle.  Its space of sections is simply the space
  of complex valued functions on $M$.  The Hermitian inner product on
  $\LL$ is defined by the standard Hermitian inner product on $\C$.

  For any real-valued 1-form $\alpha$ on $M$ the map $\nabla: \Gamma
  (TM)\times C^\infty (M, \C)\to C^\infty (M, \C)$ defined by
\[
\nabla _X f = Xf + \sqrt{-1} \alpha(X) f
\]
is a Hermitian connection.  Now consider the real-valued 1-form $\alpha$
on on $M$ given in polar coordinates $(r,\theta)$ by the equation
$\alpha = r^2 d\theta$.  A section $f\in C^\infty (M,\C)$ of $\LL$ is
covariantly constant along the polarization defined by the circles
(cf.\ Example~\ref{ex:3.punc}) if and only if
\begin{equation}\label{eq3.2}
\frac{\partial f}{\partial \theta} = -\sqrt{-1} r^2 f.
\end{equation}
A function $f$ solves the above equation if and only if it is of the form
\[
f(r,\theta) = g(r) e^{-\sqrt{-1} r^2 \theta}
\]
for some function $g(r)$.  Such an $f$ is a well-defined function on
the punctured plane only if $r^2 \in \Z$.  Thus \eqref{eq3.2} has
no nonzero smooth solutions.
\end{example}

Let us go back to the general situation: a prequantum line bundle
$(\LL, \langle \cdot, \cdot\rangle, \nabla)$ over a symplectic manifold
$(M, \omega)$ and a real polarization $\cF\subset TM$.  We now make
several assumptions:
\begin{enumerate}
\item The space of leaves $N:= M/\cF$ is a Hausdorff manifold and the
  quotient map $\pi:M\to M/\cF \equiv N$ is a submersion.
\item The space of polarized sections $\Gamma_\cF (\LL)$ is nonzero.
\end{enumerate}

\begin{remark}
  If we assume that the quotient map $\pi:M\to M/\cF$ is {\em proper}
  then it places very severe restrictions on what the connected
  components of the leaves of $\cF$ can be: they have to be compact
  tori.  See Duistermaat \cite{Duis}, for example.
In particular the fact that compact leaves of the polarization of a
punctured plane turned out to be circles (i.e., one dimensional tori)
should come as no surprise (q.v.\ Example~\ref{ex:3.punct}).
More generally the leaves of a Lagrangian fibration are open subsets
of quotients of the form $V/\Gamma$ where $V$ is a finite dimensional
real vector spaces and $\Gamma \subset V$ a discrete subgroup. That
is, $V/\Gamma \simeq (\R^k/\Z^k)\times \R^l$ for some $k$, $l$ with $k+
l = \dim V$.

The issue with existence of nonzero parallel sections reduces to the
holonomy of the connection being trivial along the leaves.  Since the
curvature of the connection vanishes identically on each leaf, the
obstruction to the existence of nonzero parallel (polarized) sections
lies in the representations of the fundamental groups of the leaves.
This is why it is not uncommon for the fundamental groups of the
leaves to be assumed away.  Fortunately there are examples of fibrations
with simply connected leaves that are slightly more general than the
cotangent bundles.  They are the so-called ``twisted cotangent
bundles'' and amount to the following.  Let $Q$ be a manifold with an
integral closed two-form $\tau$ (which may be degenerate) and let $M =
T^*Q$.  It is not hard to check that the two-form $\omega = \pi^*\tau
+\omega_{T^*Q}$ is symplectic.  Here $\pi:T^*Q \to Q$ is the canonical
projection and $\omega_{T^*Q}$ is the canonical symplectic form on
$T^*Q$ (q.v.\ Example~\ref{ex:2.5}).  The Lagrangian foliation of
$(T^*Q, \omega)$ is provided by the fibers of $\pi$, which are
contractible.
\end{remark}

\begin{definition} Given a polarization $\cF\subset TM$, the space of
  {\sf polarization preserving} functions is the space $C^\infty_\cF
  (M)$ defined by
\[
C^\infty_\cF (M):=
\{ f\in C^\infty (M)\mid [\Xi_f, X] \in \Gamma (\cF) \textrm{ for all }
X\in \Gamma (\cF)\},
\]
where, as before, $\Xi_f$ denotes the Hamiltonian vector field of the
function $f$.
\end{definition}
\begin{remark}
  It is not hard to show that if $f$ is a polarization preserving
  function then the flow of its Hamiltonian vector field $\Xi_f$
  preserves the leaves of the foliation defined by the distribution $\cF$.
\end{remark}

\begin{remark} If $\Xi_f\in \Gamma (\cF)$ then, since $\cF$ is
  involutive, $f\in C^\infty_\cF (M)$.  Using again the fact that
  $\cF$ is Lagrangian, it is not hard to show that if $f = \pi^*h$ for
  some $h\in C^\infty (M/\cF)$ then $\Xi_f \in \Gamma (\cF)$.  In
  particular the space $C^\infty_\cF (M)$ is non-trivial.
\end{remark}

\begin{lemma}
  The subspace $C^\infty_\cF(M)$ of $C^\infty (M)$ is closed under the
  Poisson bracket.
\end{lemma}

\begin{proof}
Suppose $f,g\in C^\infty_\cF(M)$ and $X\in \Gamma (\cF)$.  Then
\[
\Xi_{\{f,g\}} = [\Xi_f, \Xi_g].
\]
Hence
\[ [X, \Xi_{\{f,g\}}] = [X,[\Xi_f, \Xi_g]] = [[X, \Xi_f], \Xi_g] +
[\Xi_f, [X, \Xi_g]]
\]
where the second equality hold by the Jacobi identity.  Since $[X,
\Xi_f], [X, \Xi_g]\in \Gamma (\cF)$ by assumption, and since $\Gamma
(\cF)$ is a subspace of $\Gamma (TM)$ that is closed under brackets,
we have $[X, \Xi_{\{f,g\}}]\in \Gamma (\cF)$.
\end{proof}

\begin{lemma} If $f\in C^\infty_\cF (M)$ is a polarization-preserving
  function then the operator $Q_f$ defined by the prequantization map
  $Q:C^\infty (M) \to \End (\Gamma(\LL))$ preserves the space
  $\Gamma_\cF (\LL)$ of $\cF$ polarized sections, the sections
  covariantly constant along $\cF$.

Hence we get a map of Lie algebras
\[
Q: C_\cF^\infty (M) \to \End (\Gamma_\cF(\LL)).
\]
\end{lemma}

\begin{proof}
Note first that for any vector field $X\in \Gamma(TM)$ and any function $f\in C^\infty (M)$
\[
X(f) = df(X) = (\iota(\Xi_f)\omega) (X)  = \omega (X, \Xi_f).
\]
To prove the lemma we need to show that 
\[
\nabla _X (Q_f s) = 0
\]
for all vector fields $X\in \Gamma(\cF)$, all functions $f\in C^\infty _\cF (M)$ 
and all polarized sections  $s\in \Gamma_\cF (\LL)$.
Now
\[
\nabla_X (\nabla _{\Xi_f} s - 2\pi\sqrt{-1} fs) = 
\nabla _X(\nabla _{\Xi_f} s ) - 2\pi\sqrt{-1}X(f) \, s - 2\pi\sqrt{-1} f\nabla _X s.
\]
Note that by assumption $2\pi\sqrt{-1} f\nabla _X s=0$.  By definition
of curvature,
\[
\nabla _X(\nabla _{\Xi_f} s )= \nabla _{\Xi_f}(\nabla _X s ) + \nabla_{[X, \Xi_f]}s 
+ R^\nabla (X, \Xi_f)s.
\]
By assumption on $X$, $f$ and $s$, the first two terms are 0.  Also,
by definition of the connection
\[
R^\nabla = 2\pi \sqrt{-1} \omega.
\]
We conclude that 
\[
\nabla _X(\nabla _{\Xi_f} s ) = 0+ 0 + 2\pi \sqrt{-1} \omega (X, \Xi_f)s.
\]
Putting it all together we see that
\[
\nabla _X (Q_f s) = 2\pi \sqrt{-1} \omega (X, \Xi_f) s - 2\pi
\sqrt{-1} X(f)\, s = 2\pi \sqrt{-1} (\omega (X, \Xi_f)-\omega (X,
\Xi_f))s =0.
\]
\end{proof}

We would now like to define an inner product on the space $\Gamma_\cF
(\LL)$ of polarized sections.  If the fibers of $\pi: M\to M/\cF$ are
compact, then as before we can define an inner product on a subspace
of $\Gamma_\cF (\LL)$ consisting of square integrable sections ---
c.f. Definition~\ref{def.square-integrable} and the subsequent
discussion.  The completion of this space would give us the desired
Hilbert space.  However, as we have seen in Example~\ref{ex:3.punct},
the space of (smooth) polarized section can be 0.  This is not just an
accident of the particular example, but is fairly typical, since
fibers of proper Lagrangian fibrations are tori.  The solution to this
problem --- the lack of nonzero smooth polarized sections --- is to
consider distributional polarized sections.  We will not say anything
further on this topic in these notes.  A curious reader may consult
the discussion of distributional sections in \cite{Sniatycki}.

If the leaves of the polarization on a symplectic manifold $(M,
\omega)$ are not compact (as is the case of the vertical polarization
on a cotangent bundle $T^*N$) then none of the polarized sections are
square integrable with respect to the symplectic volume form
$\omega^m$ ($m= \frac{1}{2}\dim M)$.  On the other hand, the Hermitian inner
product $\langle s, s'\rangle$ of two polarized sections $s,s' \in
\Gamma _\cF (\LL)$ is constant along the fibers of the submersion
$\pi: M\to M/\cF$, hence descends to a function on the leaf space
$M/\cF$.  Thus it is tempting to push the function $\langle s,
s'\rangle$ down to $M/\cF$ and integrate it over the leaf space.  The
problem is that the leaf space $M/\cF$ has no preferred measure or
volume.  For instance suppose $(M,\omega)$ is a 2-dimensional
symplectic vector space $(V, \omega_V)$.  Here we think of $\omega_V$
as a constant coefficient differential form.  Then any line $\ell
\subset V$ defines a polarization whose space of leaves is the
quotient vector space $V/\ell$.  While the vector space $V/\ell$ is
one dimensional in this example, and thus isomorphic to the real line
$\R$, there is no preferred identification of $V/\ell$ with $\R$ and
no preferred measure on $V/\ell$.\\

Let us recap where we are.  We have an integral symplectic manifold
$(M,\omega)$, a prequantum line bundle $\LL\to M$ with connection
$\nabla$ whose curvature is $2\pi\sqrt{-1}\omega$, a Lagrangian
foliation $\Ff$ of $M$ so that the space of leaves $M/\Ff$ is a
Hausdorff manifold, the quotient map $\pi: M\to M/\Ff$ is a fibration
and the holonomy representation of the fundamental groups of the
leaves with respect to the connection $\nabla$ are trivial.  In this
case the prequantum line bundle $\LL\to M$ descends to a Hermitian
line bundle $\LL/\Ff\to M/\Ff$.  We now consider the new complex line
bundle $\LL/\Ff\otimes |T(M/\Ff)|^{1/2}$, the bundle $\LL/\Ff$ twisted
by the bundle of half densities on $M/\Ff$ (q.v.\
Definition~\ref{def:dens-man}).    We have an isomorphism 
\begin{equation}
\Gamma(\LL/\Ff\otimes |T(M/\Ff)|^{1/2}) \simeq \Gamma
(\LL/\Ff)\otimes_{C^\infty (M/\Ff)} \Gamma(|T(M/\Ff)|^{1/2})%  \simeq
% \Gamma_\Ff (\LL) \otimes _{C^\infty (M/\Ff)} |M/\cF|^{1/2}.
\end{equation}
of ${C^\infty (M/\Ff)}$ modules.  

It makes sense to talk about the sections of $\LL/\Ff\otimes
|T(M/\Ff)|^{1/2}$ being square integrable, and it makes sense to
define a sesquilinear pairing of two square integrable sections.
This is done as follows. As we observed in Remark~\ref{rmk:B.otimes},
the bundle $|T(M/\Ff)|^{1/2}$ of half densities is trivial, that is,
it has a nowhere zero global section.  Hence any section of
$\LL/\Ff\otimes |T(M/\Ff)|^{1/2}$ is of the form $s\otimes \mu$ where
$s\in \Gamma(\LL/\Ff) \simeq \Gamma_\Ff (\LL)$ and $\mu$ is a $1/2$
density on $M/\Ff$.  Now given a polarized section $s\in \Gamma_\Ff
(\LL)$ and a $1/2$ density $\mu$ on $M/\Ff$ we can form a 1 density
\[
\langle s, s\rangle \bar{\mu} \mu = ||s||^2 |\mu|^2
\]
on $M/\Ff$.  Moreover the map
\[
(\LL/\Ff)\otimes_{C^\infty (M/\Ff)} \Gamma(|T(M/\Ff)|^{1/2}) \to 
|M/\Ff|^1,\quad s\otimes \mu \mapsto ||s||^2 |\mu|^2
\]
is well-defined.  Here, as in Appendix~\ref{app:densities},
$|M/\Ff|^1$ denotes the space of 1-densities on the manifold $M/\Ff$,
which is the space of sections of the bundle $|T(M/\Ff)|^1$ of
1-densities.
%  We say that a section $s\otimes \mu\in
% \Gamma(\LL/\Ff\otimes |T(M/\Ff)|^{1/2}) $ is {\sf square integrable}
% if the integral $\int _{M/\Ff} ||s||^2 |\mu|^2$ of the corresponding
% 1-density converges.
It is not hard to show that the space of compactly supported polarized
sections of the twisted prequantum line bundle forms a vector space
with a Hermitian inner product given by
\[
\langle\langle s_1\otimes \mu_1, s_2 \otimes \mu_2\rangle \rangle := 
\int _{M/\Ff} \langle s_1, s_2\rangle \bar{\mu}_1 \mu_2 .
\]
The completion of this complex vector space with respect to
$\langle\langle \cdot, \cdot\rangle\rangle$ is the intrinsic quantum
space associate with the data $(\LL\to M, \nabla, \Ff, \langle \cdot,
\cdot \rangle)$.  A bit more effort gives a representation of the Lie
algebra $C^\infty_\Ff (M)$ on this quantum space.  See
\cite{SimmsWoodhouse},\cite{Sniatycki} or \cite{Woodhouse}.

\section{Prequantization of differential cocycles}

Let $L\vectb M$ be a complex line bundle, $\ipd$ a Hermitian inner
product and $\nabla$ a Hermitian connection. In Lecture 1 we were
trying to find a section of the curvature map
\begin{equation}\label{eq:5.1}
  curv:(L\vectb M, \ipd,\nabla)\longmapsto 
\frac{1}{\sqrt{-1}}R^\nabla\in\Omega^2(M),
\end{equation}
which is neither 1-1 nor onto.  Recall that a Hermitian line bundle
with connection over a manifold $M$ defines a principal $S^1$ bundle
with a connection 1-form and conversely.  Since 1-forms obviously pull
back, it will be convenient for us to replace the problem of finding
the section of \eqref{eq:5.1} by the problem of finding a section of
\begin{equation}
  curv: (S^1 \to P\to M, A\in \Omega^1(P, \R)^{S^1}) \to 
F_A \in \Omega^2 (M, \R),
\end{equation}
where $F_A$ denotes the curvature of the connection 1-form $A$.  Note
that here we think of the circle $S^1$ as $\R/\Z$ and not as the group
$U(1)$ of unit complex numbers.  Hence now our connection 1-forms are
$\R$-valued.

% For ease of presentation we replace $(L\vectb M, \ipd,\nabla)$ by a principal $S^1-bundle$ $\vcenter{\xymatrix@=1em{S^1\ar[r] & P\ar[d]\\ & M}}$ with connection 1-form $A\in\Omega^1(P,\R)^{S^1}$. Previously, $L:=(P\times\C)/S^1$ ...
% \newline

For a given manifold $M$, the collection of all principal
$S^1$-bundles with connection 1-forms over $M$ forms a category,
which we will denote by $DBS^1(M)$.\footnote{This is not a standard
  notation.  The $BS^1(M)$ is supposed to remind the reader of the
  classifying space $BS^1$, maps into which classify principal $S^1$
  bundles.  The $D$ stands for ``differential,'' i.e., the connection.
}
More precisely the objects of the category  $DBS^1(M)$ are pairs
$(P, A)$, where $P$ is a principal $S^1$ bundle over $M$ and $A\in \Omega^1(P)$ is a connection 1-form on $P$.  Given two objects $(P,A)$, $(P', A')$ of $DBS^1 (M)$ the set $\Hom((P,A),(P',A'))$ of morphisms between them is defined by 
\[
\Hom((P,A),(P',A'))=\{\phi:P\rightarrow P'\mid \phi\textrm{ is } S^1 \textrm{ 
equivariant, } \phi \textrm{ induces identity on }M,\
\phi^*A'=A\}.
\]
\noindent
Notice that all morphisms are invertible, so the category $DBS^1(M)$
is a groupoid by definition (see Appendix~\ref{sec:category}). 
% In this setting for each $(P,A)\in DBS^1(M)_0$ we want to invert
% $$(P,A)\overset{curv}\longrightarrow curv(A)\in\Omega^2(M).$$ We
% still have the same problem as before, namely this map is neither
% 1-1 nor onto. \newline\newline

Our solution of non-invertability of the curvature map proceeds along
the following lines.  We will construct a category $\Dd\Cc(M)$ of {\sf
  differential cocycles}, so that:
\begin{enumerate}
\item  The objects of $\Dd\Cc(M)$ involve differential forms.

\item There is an equivalence of categories
  $DBS^1(M)\overset{DCh}\longrightarrow\Dd\Cc(M).$
\end{enumerate}
Since equivalences of categories are invertible (up to natural
isomorphisms) this will achieve our objective.\footnote{A reader who
  is not fluent in category theory may wish at this point to
  contemplate Example~\ref{ex:equiv-of-cats} of two rather different
  looking but equivalent categories. } \, The construction of the
category $\Dd\Cc(M)$ was carried out as a toy example in a paper by
Hopkins and Singer \cite{HS}, which is where we copy the definition
from.  In constructing the functor $DCh$ we will follow
\cite{LerMalk}.  The construction of $\Dd\Cc$ requires several steps.

\subsection*{Step
  1:  Categories from cochain complexes} 

Let $A^\bullet=\{A^\bullet\overset{d}\rightarrow A^{\bullet+1}\}$ be a
cochain complex of abelian groups. For example,
$A^\bullet=\Omega^\bullet(M)$, the complex of differential forms on a
manifold $M$.  For each index $n\geq 0$ there is a category
$\Hh^n(A^\bullet)$ with the set $\{z\in A^n|dz=0\}$ of cocycles of
degree $n$ being its set of objects.  The set of morphisms
$\Hom(z,z')$ for two cocycles $z,z'$ is defined by
\[
\Hom(z,z')=\{(z,[b]) \in \ker (d:A^n\to A^{n+1}) \times A^{n-1}/
dA^{n-2} \mid z'=z+db\} \simeq \{[b] \in A^{n-1}/ dA^{n-2} \mid
z'=z+db\}
\]
% where $b\sim b+dc$ for $c\in A^{n-2}$.
The composition of morphisms is
addition $+$:
\[
(z',[b']) \circ (z,[b]) = (z,[b+b'']).
\]
The category $\Hh^n(A^\bullet)$ is a groupoid with the set
$\pi_0(\Hh^n(A^\bullet)$ of equivalence classes of objects being the
cohomology group $H^n(A^\bullet).$ The category $\Hh^n(A^\bullet)$ may
also be viewed as an action groupoid for the action of $A^{n-1}/
dA^{n-2}$ on $\ker (d:A^n\to A^{n+1})$ by way of $d: A^{n-1} \to A^n$.

Next suppose we have a contravariant functor from the category $\Man$
of manifolds and smooth maps to the category $\CoChain$ of cochain
complexes (such a functor is often called a presheaf of cochain
complexes):
\[
A^\bullet: {\Man}^{op}\to  \CoChain, \quad M\mapsto A^\bullet(M).
\]
An example to keep in mind is the functor that assigns to each
manifold the complex of differential forms and to a map of manifolds
the pullback of differential forms.  Then each smooth map
$f:M\rightarrow N$ between manifolds  induces a functor
\[
\Hh^n(f):\Hh^n(A^\bullet(N))\rightarrow \Hh^n(A^\bullet(M))
\]
with
\[
\xymatrix@=3em{z' & z\ \ \ \ \ar[l]^{(z,[b])}\ar@{|-{>}}[r]& \ \ \ \
  (f^*z' & f^*z)\ar[l]^>>>>>>{(f^*z,[f^*b])}}.
\]

\subsection*{Step 2: The presheaf of differential cocycles.}

We need to introduce more notation. Denote by
$C^\bullet(M,\Z)$ the complex of ($C^\infty$) singular integral
cochains on a manifold $M$ and by $C^\bullet(M,\R)$ the complex of
($C^\infty$) singular real-valued cochains.  We have maps of complexes
\[
C^\bullet(M,\Z) \to C^\bullet(M,\R)\quad \textrm{ and }\quad \Omega^\bullet
(M) \to C^\bullet(M,\R).  
\]
The second map sends a differential $k$-form $\sigma$ to a functional
on the space of real $k$-chains: the value of this functional on a
chain $s$ is the integral $\int_s \sigma$.  We would like to find a
complex $DC^\bullet(M)$ that fits into the diagram
\[
\xymatrix{DC^\bullet(M)\ar@{--{>}}[r]\ar@{--{>}}[d] &
  \Omega^\bullet(M)\ar[d] \\
  C^\bullet(M,\Z) \ar[r] & C^\bullet(M,\R)}.
\]
 We define this complex as follows:
\[
DC^k(M)=\{(c,h,\omega)\in C^k(M,\Z)\times C^{k-1}(M,\R)\times\Omega^k(M)
\mid \omega=0\ for\ k<2\}
\]
with the differential $\tilde{d}$ defined by
\[
\tilde{d}(c,h,\omega)=(\delta c,\omega-c-\delta h, d\omega).
\]
Note that in defining $\tilde{d}$ we suppressed the maps
$C^\bullet(M,\Z) \to C^\bullet(M,\R)$ and $\Omega^\bullet (M) \to
C^\bullet(M,\R)$.  In particular, when we think of a differential
$k$-form $\tau$ as a real cochain, we write its differential as $\delta
\tau$.  That is, for any $k+1$ chain $\gamma$ we have
\[
\delta \tau (\gamma) = \int _\gamma d\tau.
\]
It is not hard to show that
\[
\tilde{d}\circ \tilde{d}=0.\]
Indeed, 
\[
\tilde{d}(\tilde{d}(c,h,\omega)) = \tilde{d}(\delta c,\omega -c -
\delta h,d\omega)= (\delta^2 c, d\omega - \delta c -\delta (\omega -c
-\delta h), d^2 \omega) = 
%(0, d\omega -\delta c -d\omega + \delta c +\delta^2 h, 0) = 
(0,0,0).
\]

We are now ready to define the category $\Dd\Cc(M)$ of differential
cocycles by setting 
\[
\Dd\Cc(M):=\Hh^2(DC^\bullet(M)).
\]
By construction the set of objects $\Dd\Cc(M)_0$ of this category is
\[
\Dd\Cc(M)_0=\{(c,h,\omega)\in C^2(M,\Z)\times C^1(M,\R)\times \Omega^2(M)
\mid \delta c=0,\ d\omega=0,\ \omega=c-\delta h\}
\]
 and the morphisms are defined by
\[
\Hom((c,h,\omega),(c',h',\omega'))
=\left\{[e,k,0] \, \left| \,
e\in C^1(M,\Z),
\ k\in C^0(M,\R)\,
\textrm{ and }\, \begin{array}{c} c'-c=\delta e\\ h'-h=-\delta k-e\\ \omega-\omega'=0 \end{array}\right. \right\}.
\]

% \underline{Remark:} We insist that $\omega-\omega'=0$ so that morphisms do not change curvature.\newline\newline
\noindent
The following theorem then holds (it is presented as a warm-up example
in \cite{HS}):

\begin{Theorem}
 For each manifold $M$ there exists an equivalence of categories
\[ 
DCh_M:DBS^1(M)\rightarrow\Dd\Cc(M)
\]
with $DCh_M(P,A)=(c(P,A), h(P,A), F_A)$ for each principal circle
bundle with connection $(P, A)$.  Here as before $F_A$ denotes the
curvature of a connection $A$.
\end{Theorem}
\begin{remark}
  The functor that assigns a prequantum line bundle to a differential
  cocycle is the ``homotopy inverse" of the functor $DCh_M$.
\end{remark}

Finally here is an outline of an argument as to why this theorem is
true and how you would go about writing down the functor $DCh$.  I
will be following the presentation in \cite{LerMalk}.  Here are the
main ideas:

\begin{enumerate}
\item  Do it for all manifolds at once.

\item Restate the theorem for presheaves of categories: 

\begin{quote}There exists a
  morphism $DCh_{(\cdot )}:DBS^1(\cdot)\rightarrow\Dd\Cc(\cdot)$ of
  presheaves of categories with the desired properties. 
\end{quote}
Here the
  $\cdot$ is a place holder for a manifold.

\item Write down the functor $DCh$ explicitly for the sub-presheaf of
  trivial bundles $DBS^1_{triv}(\cdot)\subset DBS^1(\cdot)$.  This is
  not hard.  The set of objects of the presheaf $DBS^1_{triv}(\cdot)$
  on a manifold $M$ is the set
\[
DBS^1_{triv}(M)_0=\{(M\times S^1,a+d\theta)|a\in\Omega^1(M)\},
\]
The set of morphisms is 
\[
\Hom((M\times S^1,a+d\theta),(M\times S^1,a'+d\theta))=\{f:M\rightarrow S^1|a=a'+f^*d\theta\}.
\]
For each manifold $M$ we therefore define the functor $DCh_M: DBS^1_{triv}(M) \to  \Dd\Cc(M)$ as follows: on objects
\[
DCh_M (M\times S^1,a+d\theta) =(0,a,da),
\]
on morphisms
\[
DCh_M (f:M\rightarrow S^1) = [\delta(\tilde f) -f^*d\theta,\tilde f,0],
\]
where $\tilde f:M\rightarrow\R$ is {\sf any} lift of $f$ (not
necessarily continuous). We think of $\tilde f$ as a real $0$-cochain:
\[
\tilde f\left(\sum\limits_{p\in M}n_p\right)=\sum n_p\tilde f(p)
\]
for any zero chain $\sum\limits_{p\in M}n_p$.
\end{enumerate}

The rest of the argument uses the following facts (see \cite{LerMalk}
for details):
\begin{enumerate}
\item The functor 
 $DCh_M:DBS^1_{triv}(M)\rightarrow\Dd\Cc(M)$ is bijective on $\Hom$'s 
(that is,  it is a fully faithful functor).

\item If $M$ is contractible (e.g. an open ball) then $DCh_M$ is
  essentially surjective.  Of course, for a {\sf general} manifold $M$
  not all bundles over $M$ are trivial, but they are all {\sf locally}
  trivial.

\item Any bundle with a connection can be glued together out of the
  trivial ones.

\item Differential cochains glue like bundles. Showing this requires work.
\end{enumerate}
These last two facts amount to saying that the two presheaves
$DBS^1(\cdot)$ and $\Dd\Cc(\cdot)$ are stacks over $\underline{Man}$,
the category of Manifolds. 

The presheaf $DBS^1_{triv}$ is {\sf not} a stack since gluing a bunch
of trivial bundles together need not result in a trivial bundle.
There is an operation on presheaves of categories called
stackification. This is a version of sheafification for sheaves of
categories. The stackification of $DBS^1_{triv}$, not surprisingly, is
$DBS^1$.  Therefore, by the universal property of stackifications
there is a unique functor $DCh: DBS^1 (\cdot) \to \Dd\Cc (\cdot)$
making  the following diagram 
\[
\xymatrix{DBS^1(\cdot)\ar@{--{>}}[rd]^{\exists !} & \\& \Dd\Cc (\cdot)
\\ DBS^1_{triv}(\cdot)\ar@{^{(}-{>}}[uu]\ar[ru]_{DCh} & }
\]
commute.  The resulting functor
\[ 
DCh: DBS^1 (\cdot) \to \Dd\Cc (\cdot)
\]
is an equivalence of categories since it is an equivalence of
categories locally.

\appendix

\section{Elements of  category theory}\label{sec:category}

\subsection{Basic notions}
\mbox{}\\[-8pt]

\noindent
We start by recalling the basic definitions of category theory, mostly
to fix our notation.  This appendix may be useful to the reader with
some background in category theory; the reader with little to no
experience in category theory may wish to 
consult a textbook such as \cite{Awodey}.

\begin{definition}[Category]\label{def:category}
  A {\bf category} $\A$ consists of

\begin{enumerate} 
\item A collection\footnote{A collection may be too big to be a set;
    we will ignore the  set-theoretic
    issues this may lead to.}  $\A_0$ of {\em objects}; 

\item For any two objects $a,b\in \A_0$, a set $\Hom_\A (a,b)$ of of
  {\em morphisms} (or {\em arrows});
\item For any three objects $a,b,c\in \A_0$, and any two arrows
  $f\in\Hom_\A(a, b)$ and $g\in\Hom_\A(b,c)$, a {\em composite }
  $g\circ f\in\Hom_\A(a,c)$, i.e., for all triples of objects
  $a,b,c\in\A_0$ there is a {\em composition map}
\[
\circ\colon \Hom_\A(b,c)\times\Hom_\A(a,b)\rightarrow\Hom_\A(a,c),
\]
\[
\Hom_\A(b,c)\times\Hom_\A(a,b) \ni (g,f)\mapsto g\circ f\in \Hom_\A(a,c).
\] 
This composition operation is {\em
associative} and has {\em units}, that is,
\begin{itemize} 
\item[i.] for any triple of morphisms $f\in\Hom_\A(a,b)$,
  $g\in\Hom_\A(b,c)$ and $h\in\Hom_\A(c,d)$ we have
\[ 
h\circ (g\circ f)=(h\circ g)\circ f\,; 
\] 
\item[ii.] for any object $a\in \A_0$, there exists a morphism
  $1_a\in\Hom_\A(a,a)$, called the {\em identity}, which is such that
  for any $f\in\Hom_\A(a,b)$ we have
\[ 
f=f\circ
1_a=1_b\circ f\,.  
\] 
\end{itemize}
We denote the collection of all morphisms of a category $\A$ by $\A_1$:
\[
\A_1 = \bigsqcup _{a,b\in \A_0} \Hom_\A (a,b).
\]
\end{enumerate} 
\end{definition}
\begin{remark}
  The symbol ``$\circ$'' is customarily suppressed in writing out
  compositions of two morphisms.  Thus
\[
gf\equiv g\circ f.
\]
\end{remark}

\begin{example}[Category $\Set$ of sets] 
The collection $\Set$ of all sets forms a category.
  The objects of $\Set$ are sets, the arrows of $\Set$ are ordinary
  maps and the composition of arrows is the composition of maps.
\end{example}

\begin{example}[Category $\Vect$ of vector spaces]
  The collection $\Vect$ of all real vector spaces (not necessarily
  finite dimensional) forms a category.  Its objects are vector spaces
  and its morphisms are linear maps.  The composition of morphisms is
  the ordinary composition of linear maps.
\end{example}

\begin{example} [The category $\M$ of coordinate vector spaces] The
  objects of this category are coordinate vector spaces $0 =
  \R^0,\R^1, \ldots, \R^n \ldots$.  The set of morphism from $\R^m$ to
  $\R^n$ is the set of all $n\times m$ matrices.  The composition of
  morphisms is given by a matrix multiplication.
\end{example}
\begin{remark} For a category $\A$ there are two maps from the
  collection $\A_1$ of its arrows to the collection $\A_0$ of objects
  called {\sf source} and {\sf target} and denoted respectively by $s$
  and $t$.  They are defined by requiring that 
\[
s(f) =a \quad \textrm{and} \quad t(f) = b \quad \textrm{for any  } f\in
\Hom_\A (a,b).
\]
\end{remark}

\begin{definition}\label{def:subcat} 
  A {\em subcategory} $\A$ of a category $\B$ is a collection of some
  objects $\A_0$ and some arrows $\A_1$ of $\B$ such that:
\begin{itemize}
\item For each object $a\in \A_0$, the identity $1_a$ is in $\A_1$;
\item For each arrow $f\in \A_1$ its source and target $s(f), t(f)$
  are in $\A_0$;
\item for each pair $(f,g)\in \A_0 \times \A_0$ of composable arrows
  $a\stackrel{f}{\to}a'\stackrel{g}{\to}a''$ the composite $g\circ f$
is in $\A_1$ as well.
\end{itemize}
\end{definition}
\begin{remark}
Naturally a subcategory is a category in its own right.
\end{remark}

\begin{example}
  The collection $\Finset$ of all finite sets and all maps between
  them is a subcategory of $\Set$ hence a category.
 The collection $\FinVect$ of real finite dimensional vector spaces
  and linear maps is a subcategory of $\Vect$.
\end{example}

\begin{example}
  A subcategory $\FinVect^{iso}$ is defined to have the same objects
  as the category of $\FinVect$.  Its morphisms are {\sf isomorphisms }
  of vector spaces.  Since the composition of two linear isomorphisms
  is an isomorphism $\FinVect^{iso}$ is a subcategory of $\FinVect$.

  Note that for any object $V$ in $\FinVect^{iso}$, that is, for any
  finite dimensional vector space $V$, the set of morphisms
  $\Hom(V,V)$ in $\FinVect^{iso}$ is $GL(V)$, the Lie group of
  invertible linear maps from $V$ to $V$.

  Compare this to the fact that in the category $\FinVect$ the set of
  morphisms $\Hom (V,V)$ is $\End(V)$, the space of {\sf all} linear
  maps from $V$ to itself.
\end{example}

\begin{definition}[isomorphism]
  An arrow $f\in \Hom_\A(a,b)$ in a category $\A$ is an {\em
    isomorphism} if there is an arrow $g\in \Hom_A (b,a)$ with $g\circ
  f = 1_a$ and $f\circ g = 1_b$.  We think of $f$ and $g$ as inverses
  of each other and may write $g = f\inv$.  Clearly $g= f\inv$ is also
  an isomorphism.

  Two objects $a,b\in \A_0$ are {\em isomorphic} if there is an
  isomorphism $f\in \Hom_\A (a,b)$.  We will also say that $a$ is
  isomorphic to $b$.
\end{definition}

\begin{definition}[Groupoid] \label{def:groupoid} A {\em groupoid} is
  a category in which every arrow is an isomorphism.
\end{definition}
\begin{example}
  The category $\FinVect^{iso}$ is a groupoid. 
\end{example}

\begin{definition}[Functor]\label{def:functor}
  A (covariant) {\em functor} $F\colon\A\to\B$ from a category $\A$ to
a category $\B$ is a map on the objects and arrows of $\A$ such that
every object $a\in \A_0$ is assigned an object $Fa\in\B_0$, every
arrow $f\in\Hom_\A(a,b)$ is assigned an arrow $Ff\in\Hom_\B(Fa,Fb)$,
and such that composition and identities are preserved, namely
\begin{equation*}
  F(f\circ g) = Ff\circ Fg,\quad F 1_a = 1_{Fa}.
\end{equation*}
A {\em contravariant} functor $G$ from $\A$ to $\B$ is a map on the
objects and arrows of $\A$ such that every object $a\in \A_0$ is
assigned an object $Ga\in\B_0$, every arrow $f\in\Hom_\A(a,b)$ is
assigned an arrow $Gf\in\Hom_\B(Gb,Ga)$ (note the order reversal),
such that identities are preserved, and the composition of arrows is
reversed:
\[
G (f\circ g) = G(g)\circ G(f)
\] 
for all composable pairs of arrows $f,g$ of $\A$.
\end{definition}

\begin{example} \label{ex:iota}
  There is a natural functor $\iota:\M \to \FinVect$ which is the
  identity on objects and maps an $n\times m$ matrix to the
  corresponding linear map.
\end{example}
\begin{example} The functor
  $(-)^*:\mathsf{FinVect}\to\mathsf{FinVect}$ that takes the duals,
  that is,
  $(V\overset{A}{\longrightarrow}W)\mapsto
(V^*\overset{A^*}{\longleftarrow}W^*)$
  is a contravariant functor.
\end{example}
% \begin{example}
%   $[(-)^*]^{-1}:\mathsf{FinVect}^{iso}\to\mathsf{FinVect}^{iso}$
%   defined by
%   $(V\overset{A}{\longrightarrow}W)\mapsto(V^*\overset{(A^*)^{-1}}{\longrightarrow}W^*)$
%   is a (covariant) functor.
% \end{example}

\begin{remark}
Since functors are maps, functors can be composed.
\end{remark}

\begin{definition}\label{def:full}
A functor $F\colon \A\to \B$ is
\begin{enumerate}
\item {\em full} if $F\colon\Hom_\A(a,a')\to\Hom_\B(Fa,Fa')$
is surjective for all pairs of objects $a,a'\in \A_0$;
\item {\em faithful} if $F\colon\Hom_\A(a,a')\to\Hom_\B(Fa,Fa')$
is injective for all pairs of objects $a,a'\in \A_0$
\item {\em fully faithful} if $F\colon\Hom_\A(a,a')\to\Hom_\B(Fa,Fa')$
  is a bijection for all pairs of objects $a,a'\in \A_0$;
\item {\em essentially surjective} if for any object $b\in \B_0$ there
  is an object $a\in \A_0$ and an isomorphism $f\in \Hom_\B (F(a),
  b)$.  That is, for any object $b$ of $\B$ there is an object $a$ of
  $\A$ so that $b$ and $F(a)$ are isomorphic.
\end{enumerate}
\end{definition}

\begin{example}
  The functor $\iota: \M\to \FinVect$ is fully faithful (since any
  linear map from $\R^n$ to $\R^m$ is uniquely determined by what it
  does on the standard basis) and essentially surjective (since any
  real vector space of dimension $n$ is isomorphic to $\R^n$).
\end{example}

\begin{definition}[Natural Transformation]
  Let $F,G\colon\A\to\B$ be a pair of functors. A {\em natural
    transformation} $\tau\colon F\Rightarrow G$ is a family of
  $\{\tau_a\colon Fa\to Ga\}_{a\in\A_0}$ of morphisms in $\B$, one for
  each object $a$ of $\A$, such that, for any $f\in\Hom_\A(a,a')$, the
  following diagram commutes:
\begin{equation*}
  \xy
(-10, 10)*+{Fa} ="1"; 
(10, 10)*+{Fb} ="2"; 
(-10,-5)*+{Ga}="3";
(10,-5)*+{Gb} ="4";
{\ar@{->}^{Ff} "1";"2"};
{\ar@{->}^{\tau_b} "2";"4"};
{\ar@{->}_{\tau_a} "1";"3"};
{\ar@{->}^{Fg} "3";"4"};
\endxy
\end{equation*}
If each $\tau_a$ is an isomorphism, we say that $\tau$ is a {\em natural
isomorphism} (an older term is {\em natural equivalence}).
\end{definition}

\begin{definition}[Equivalence of
categories]
An {\em equivalence of categories} consists of a pair of functors
\[
F\colon \A\to \B, \quad E\colon \B\to \A
\]
and a pair of natural isomorphisms
\[
\alpha\colon  1_\A \Rightarrow E\circ F \quad \beta\colon  1_\B \Rightarrow F\circ E.
\]
\end{definition}
In this situation the functor $F$ is called {\em the pseudo-inverse}
or the {\em homotopy inverse} of $E$.  The categories $\A$ and $\B$
are then said to be {\em equivalent categories}.

\begin{proposition}\label{prop:eq-of-cats}
  A functor $F\colon \A \to \B$ is (part of) an equivalence of categories if
  and only if it is fully faithful and essentially surjective.
\end{proposition}
\begin{proof}
See \cite[Proposition~7.25]{Awodey}
\end{proof}

\begin{example}\label{ex:equiv-of-cats}
  The categories $\M$ of matrices and $\FinVect$ of finite dimensional
  vector spaces are equivalent categories since the functor $\iota:\M
  \to \FinVect$ (q.v.\ Example~\ref{ex:iota}) is fully faithful and
  essentially surjective.  Note that the functor $\iota$ is {\em not}
  surjective on objects.
\end{example}

\section{Densities}\label{app:densities}

In this section we borrow from a manuscript by Guillemin and Sternberg
\cite{GSsemiclassical}.

\subsection{Densities on a vector space}

Consider an $n$-dimensional real vector space $V$.  Recall that a
basis $\{v_1, \ldots, v_n\}$ defines a linear isomorphism
$\mathsf{v}:\R^n \to V$ by $\mathsf{v}(x_1, \ldots, x_n) = \sum x_i
v_i$.  Conversely any linear isomorphism $\mathsf{v}:\R^n \to V$
defines a basis $\{v_1, \ldots, v_n\}$ of $V$ by setting the $i^{th}$ basis
vector $v_i$ to be the image $\mathsf{v}(e_i)$ of the standard
$i^{th}$ basis vector $e_i$ of $\R^n$ under the isomorphism
$\mathsf{v}$.  From now on we will not distinguish between a basis of
$V$ and a linear isomorphism from $\R^n$ to $V$.

\begin{definition}[Frame] We denote the space of bases of an $n$
  dimensional real vector space $V$ by $\Fr(V)$ and refer to it as the
  space of {\sf frames} of $V$.
\end{definition} 

Note that there is a natural {\sf right} action $\bullet$ of the Lie
group $GL(n, \R) \equiv GL(\R^n)$ on the space of frames $\Fr(V)$ of
an $n$-dimensional vector space $V$ by composition on the right:
\[
\mathsf{v}\bullet A : = \mathsf{v}\circ A
\]
for all isomorphisms $\mathsf{v}:\R^n \to V$ and all $A\in GL(n, \R)$.
Moreover this action is {\sf free}  and {\sf transitive}: given
$\mathsf{v}, \mathsf{v}'\in \Fr(V)$ and $A \in GL(n, \R)$
\[
\mathsf{v}\bullet A  =\mathsf{v} \quad \textrm{ if and only if }
A =(\mathsf{v}\inv \circ \mathsf{v}').
\]
\begin{remark}
  A space $X$ with a free and transitive action of a group $G$ is
  called a $G$ {\sf torsor}.
\end{remark}
\noindent
With these preliminaries out of the way we are ready to define
$\alpha$-densities on a vector space.

\begin{definition}[$\alpha$-density] Let $\alpha$ denote a complex
  number.  An {\sf $\alpha$-density} (also called a {\sf density of
    order $\alpha$}) on an $n$-dimensional real vector space $V$ is a
  map
\[
\tau : \Fr (V) \to \C \quad \textrm{ with } \quad  
\tau (\ve\bullet A ) = \tau(\ve)|\det(A)|^\alpha
\]
for all $\ve\in \Fr(V)$ and all $A\in GL(n, \R)$.
\end{definition}
\begin{notation}
  Since $\alpha$-densities on a fixed vector space $V$ are complex
  valued functions, they form a complex vector space.  We denote it by
  $|V|^\alpha$. In other words
\[
|V|^\alpha := \{ \tau :\Fr (V) \to \C \mid \tau (\ve\bullet A ) =
\tau(\ve)|\det(A)|^\alpha \textrm{ for all } \ve\in \Fr(V), A\in GL(n,
\R)\}.
\]
\end{notation}

\begin{remark}
  Alternatively one may view the space of frames $\Fr(V)$ as a
  principal $GL(n,\R)$ bundle over a point.  An $\alpha$-density is
  then a section of the associated bundle $(\Fr(V)\times
  \C)/GL(n,\R)$, where $GL(n,\R)$ acts on $\C$ by the character
  $A\mapsto |\det A|^\alpha$.  Since $(\Fr(V)\times \C)/GL(n,\R)$ is a
  complex line bundle over a point, it is a one dimensional complex
  vector space.  Hence the space of sections of this bundle, i.e., the
  space of densities $|V|^\alpha$, has complex dimension 1. In
  particular we have proved:
\end{remark}

\begin{lemma} The space $|V|^\alpha$ of $\alpha$-densities on a vector
  space $V$ is a complex 1-dimensional vector space.
\end{lemma}

\begin{remark}\label{rm:B.8}
  Any {\sf nonzero} $n$-form $\omega \in \Lambda^n(V^*)$ on an
  $n$-dimensional real vector space $V$ defines an $\alpha$-density
  $|\omega|^\alpha$ by the formula
\[
|\omega|^\alpha (v_1, \ldots, v_n) := |\omega(v_1, \ldots,  v_n)|^\alpha
\]
for all frames $\{v_1,\ldots, v_n\}\in \Fr(V)$.  

Conversely, since the space $|V|^\alpha$ of $\alpha$ densities is
1-dimensional, any density $\tau\in |V|^\alpha$ is of the form $\tau =
c |\omega|^\alpha$ for some nonzero $n$-form $\omega\in \Lambda^n
(V^*)$ and a constant $c\in \C$.
\end{remark}

\begin{remark}
  Densities pull back under linear isomorphisms.  If $T:W\to V$ is an
  isomorphism of $n$-dimensional vector spaces and $\tau \in
  |V|^\alpha$ is a density then its pullback $T^*\!\tau:\Fr(W)\to \C$ is
  defined by
\[
T^*\!\tau \,(w_1, \ldots w_n) = \tau (Tw_1, \ldots Tw_n)
\]
for any frame $\mathsf{w} = \{w_1,\ldots, w_n\}\in \Fr(W)$. Note that
since $T$ is an isomorphism the tuple $ (Tw_1, \ldots Tw_n)$ {\em is}
a frame of $V$, so the definition of pullback does make sense.
\end{remark}

\begin{remark}
  Densities can be multiplied: if $\rho \in |V|^\alpha$ and $\tau \in
  |V|^\beta$ are two densities of order $\alpha$ and $\beta$
  respectively then their product defined by 
\[
(\rho \cdot \tau ) (\ve) = \rho (\ve) \tau (\ve)
\]
is easily seen to be a density of order $\alpha +\beta$.  Since the
multiplication map $(\rho, \tau) \mapsto \rho\cdot \tau$ is $\C$ bilinear, we
get a $\C$ linear map
\[
|V|^\alpha \otimes  |V|^\beta \to |V|^{\alpha + \beta}.
\]
Since the map is nonzero, it is an isomorphism of vector spaces by
dimension count.  In particular we have a canonical isomorphism
\[
|V|^{1/2} \otimes  |V|^{1/2} \stackrel{\simeq}{\to} |V|^1.
\]
\end{remark}

\begin{remark}\label{rmk:B.11}
  It makes sense to take a complex conjugate of a density.  If $\rho
  \in |V|^\alpha$ is an $\alpha$-density then $\bar{\rho}:\Fr(V)\to
  \C$ defined by
\[
\bar{\rho} (\ve) := \conj{ \rho (\ve)}
\]
for all $\ve\in \Fr(V)$ is easily seen to be an $\bar{\alpha}$-density.

Therefore  we can define on the space $|V|^{1/2}$ of half-densities a 
$ |V|^1$-valued Hermitian inner product by
\[
(\mu, \tau):= \bar{\mu}\tau.
\]
\end{remark}

\subsection{Densities on manifolds}

Recall that for any real vector bundle $E\to M$ of rank $k$ over a
manifold $M$ we have the principal $GL(k, \R)$ bundle $\Fr (E)\to M$,
the so-called {\sf frame bundle} of $E\to M$.  A typical fiber
$\Fr(E)_q$ of this bundle above a point $q$ of $M$ consists of the
frames of the fiber $E_q$ of the bundle $E$:
\[
\Fr(E)_q := \Fr(E_q).
\]
Recall also that we can think of $\Fr(E)$ as an open subset of the
vector bundle $\Hom (M\times \R^k, E)\to M$; $\Fr(E)$ consists of
isomorphisms.  There is a natural right $GL(k,\R)$ action of $\Fr(E)$
making it a principal $GL(k,\R)$ bundle.

Next recall that given a principal $G$-bundle $G\to P\to M$ over a
manifold $M$ and a (complex) representation $\rho:G\to GL(W)$, we can
build out of this data a complex vector bundle $P\times ^\rho W \to M$
over $M$.  This bundle is the quotient of the trivial bundle $P\times
W$ by a free and proper left action of $G$:
\[
P\times ^\rho W : = (P\times W)/G,\quad g\cdot (p, w):= (pg\inv,
\rho(g)w).
\]
It will be useful to recall that the space of sections of the
associated bundle $\Gamma (P\times ^\rho W)$ is  isomorphic to
the space of equivariant $W$-valued functions on $P$:
\[
\Gamma (P\times ^\rho W) \simeq \{\varphi:P\to W\mid \varphi(p\cdot g)
= \rho (g)\inv \varphi (p)\}.
\]

\begin{definition}[densities on a manifold] \label{def:dens-man}
  We define the complex line bundle $|TM|^\alpha\to M$ of
  $\alpha$-densities on a manifold $M$ to be the associated bundle
\[
|TM|^\alpha := \Fr(TM)\times ^{|\det|^{-\alpha}} \C,
\]
where the representation  $|\det|^{-\alpha}:GL(k,\R)
\to GL(\C)$ is defined by
\[
|\det |^{-\alpha} (A)z := |\det A|^{-\alpha}z \quad \textrm{ for all
}\quad A\in GL(k,\R), z\in \C.
\]
We refer to the sections of the bundle $|TM|^\alpha\to M$ as {\sf
  $\alpha$-densities on the manifold $M$}.  We denote the space of
$\alpha$-densities on a manifold $M$ by $|M|^\alpha$:
\[
|M|^\alpha := \Gamma (\Fr(TM)\times ^{|\det|^{-\alpha}} \C).
\]
\end{definition}

\begin{remark}
  We may and will identify the space of $\alpha$-densities on a
  manifold $M$ with the space of equivariant complex functions on its
  frame bundle:
\[
|M|^\alpha \simeq \{ \tau: \Fr(TM)\to \C\mid \tau (\ve\bullet A) =
|\det A|^\alpha \tau (\ve)\,\, \textrm{ for all } A\in GL(m,\R),
\ve\in \Fr(TM)\}.
\]
\end{remark}
\begin{remark}
  Note that by design the fiber of the bundle $|TM|^\alpha\to M$ at a
  point $q$ is the 1-dimensional complex vector space $|T_q M|^\alpha$
  of $\alpha$-densities on the tangent space $T_q M$.
\end{remark}

\begin{remark} \label{rmk:B.otimes}
  By construction the transition maps for the complex line
  bundle $|TM|^\alpha \to M$ take their values in positive real
  numbers.  It follows that $|TM|^\alpha \to M$ is a trivial bundle
  but not canonically.  In particular its space of sections
  $|M|^\alpha$ is a rank 1 $C^\infty(M, \C)$ module.   Since the space of
sections of the tensor product is isomorphic to the tensor product of 
sections (as $C^\infty(M)$ modules),  it follows that  for
  any complex line bundle $\LL \to M$, a section of $\LL\otimes
  |TM|^\alpha$ is of the form $s\otimes \tau$ for some section $s\in \Gamma
  (\LL)$ of $\LL\to M$ and some  $\alpha$-density $\tau $.
\end{remark}

\begin{remark}
  Since densities on a vector space pull back under a linear
  isomorphism, densities on a manifold pull back under local
  diffeomorphisms: If $F:N\to M$ is a local diffeomorphism between two
  manifolds and $\tau :\Fr(TM)\to \C$ an $\alpha$-density on $M$, the
  pullback $F^*\tau \in |N|^\alpha$ is defined as follows: for any
  point $q\in N$ and any frame $(v_1,\ldots, v_n)$ of $T_q N$
\[
(F^*\tau)_q (v_1,\ldots, v_n) := \tau _{F(q)} (dF_q v_1,\ldots, dF_q v_n).
\]
\end{remark}

\begin{remark}
  For any open subset $U$ of $\R^m$ we have the canonical
  $\alpha$-density $|dx_1\wedge \cdots dx_m|^\alpha$ (q.v.\
  Remark~\ref{rm:B.8}).  Here, of course, $x_1, \ldots, x_m$ are the
  Cartesian coordinates on $\R^m$.  This density defines an
  isomorphism $U\times \C \to |TU|^\alpha$. Therefore, for any
  $\alpha$-density $\tau$ on $U$ there exists a unique function
  $f_\tau \in C^\infty (U)$ with
\[
 \tau = f_\tau |dx_1\wedge \cdots dx_n|^\alpha.
\]
Indeed,
\[
f_\tau = \tau \left(\frac{\partial}{\partial x_1},\ldots,
  \frac{\partial}{\partial x_m}\right).
\]
\end{remark}
Putting the preceding remarks together we note that for a
diffeomorphism $F:U\to V$ of open subsets of $\R^m$ and an
$\alpha$-density $f(y) |dy_1\wedge \ldots \wedge dy_m|^\alpha$ on $V$
we have
\begin{equation} \label{eq:B1}
 F^* \left( f(y) |dy_1\wedge \ldots \wedge dy_m|^\alpha  \right) =
f(F(x))|\det dF_x|^\alpha  |dx_1\wedge \cdots\wedge dx_n|^\alpha.
\end{equation}
Formula \eqref{eq:B1} is a reason why 1-densities can be integrated
over manifolds.  The story parallels the familiar story of integration
of top degree forms on oriented manifolds.  The 1-densities have the
advantage that the manifold doesn't have to be oriented (or even be
orientable) for their integrals to make sense.

The story proceeds as follows. If $U\subset \R^n$ is
an open set and $\tau \in |U|^1$ a 1-density, then, as remarked above,
\[
\tau = f_\tau |dx_1\wedge \ldots\wedge dx_n|
\]
for a unique complex valued smooth function $f_\tau$ on $U$.  One
defines the integral $\int_U \tau $ by
\[
\int_U\tau := \int _U f_\tau.
\]
Or, if you prefer,
\[
\int _U  f \, \, |dx_1 \wedge\ldots\wedge dx_n|: = \int _U f \, \, dx_1\ldots dx_n
\]
for any smooth integrable function $f$ on the set $U$.
%\in C^\infty (U)\cap L^1 (U)$.
This is a complex valued integral.

If $M$ is a manifold of dimension $m$, $\phi %  =(x_1, \ldots, x_m)
:U\to \R^m$ a coordinate chart and $\tau$ is a 1-density on $M$ with support
in $U$ one defines
\[
\int _M \tau := \int _{\phi(U)} (\phi\inv)^* \tau.
\]
If $\psi: U\to \R^m$ is another coordinate chart, then
\[
(\psi \inv)^*\tau = f(y) \,|dy_1\wedge \cdots\wedge dy_m|
\]
for some $f\in C^\infty (\psi(U))$.  Consider the diffeomorphism $F =
\psi \circ \phi\inv : \phi(U) \to \psi(U)$.  We have
\[
(\phi\inv)^*\tau =(\psi\inv \circ \psi \phi\inv)^*\tau = F^*(\psi\inv)^*\tau.
\]
By \eqref{eq:B1}
\[
(\phi\inv)^*\tau = f(F(x)) |\det DF_x| |dx_1\wedge\cdots\wedge dx_m|. 
\]
Therefore
\[
\int _{\phi (U) }(\phi\inv)^*\tau = \int _{\phi (U) } f(F(x))\, |\det
DF_x|\, dx_1\cdots dx_m = \int_{F(\phi(U))=\psi (U)} f(y)\, dy_1\cdots
dy_m = \int_{\psi(U)} (\psi\inv)^*\tau,.
\]
where the second equality holds by the change of variables formula for
functions on regions of $\R^m$.  Therefore integrals of 1-densities
supported in coordinate charts are well-defined.

Next one makes sense of integrability of {\sf non-negative} densities.
A 1-density $\tau$ on a manifold $M$ is {\sf non-negative} if its
value at any frame $\{v_1, \ldots v_m\} \subset T_q M$ is a
non-negative real number.  It is not hard to see that
``non-negativity'' is well-defined.  Now choose a locally finite cover
$\{U_\alpha\}$ of $M$ by coordinate charts and choose a partition of
unity $\{\rho_\alpha\}$ subordinate to the cover.  We say that a
non-negative 1-density is integrable if the sum
\[
\sum _\alpha \int _M \rho_\alpha \tau =\sum _\alpha \int _{U_\alpha}
\rho_\alpha \tau
\]
converges.  We then define $\int_M \tau$ to be the sum:
\[
\int_M \tau := \sum _\alpha \int _{U_\alpha}
\rho_\alpha \tau.
\]
The rest of the definition of integration of 1-densities proceeds just
like for functions.  A given real-valued 1-density $\tau$ can be
written as a difference of two (continuous) non-negative 1-densities
$\tau_+$ and $\tau_-$:
\[
\tau = \tau_+ - \tau_-.
\] 
We call $\tau$ integrable if $\int_M \tau_+$ and $\int_M\tau_-$ are
finite.  We then set $\int_M\tau = \int_M \tau_+ - \int_M\tau_-$.
Finally we define a complex valued 1-density $\tau$ integrable if its
real and imaginary parts are integrable; we define its integral to be
the sum
\[
\int _M\tau = \int_M Re(\tau) + \sqrt{-1}\int_M Im(\tau).
\]

\subsection{The ``Intrinsic'' Hilbert space} \label{subsec:intrHilb}

Suppose $(\LL\to M, \langle
\cdot, \cdot\rangle)$ is a Hermitian line bundle.  Then given a
1-density $\tau$ on $M$ we can define a Hermitian pairing pairing of
sections of $\LL$ by
\[
\langle\langle s, s'\rangle\rangle := \int _M \langle s, s'\rangle \tau.
\]
There is an associated Hilbert space of ``square integrable''
sections, which, of course, depends on the choice of our density
$\tau$.  There is also a more intrinsic pairing of sections of a
slightly different bundle.  Consider the tensor product of complex
line bundles $\LL\otimes |M|^{1/2} \to M$. By
Remark~\ref{rmk:B.otimes} a section of $\LL\otimes |M|^{1/2} \to M$ is
of the form $s\otimes \mu$, where $s\in \Gamma (\LL)$ and $\mu$ is a
$\frac{1}{2}$-density.  Now given two sections $s_1\otimes \mu_1$ and
$s_2\otimes\mu_2$ we can pair them to get a 1-density $\langle s_1,
s_2\rangle \bar{\mu}_1\mu_2$ (q.v.\ Remark~\ref{rmk:B.11}).
  Hence the Hermitian inner product
\[
\Gamma(\LL\otimes |M|^{1/2})\times  \Gamma(\LL\otimes |M|^{1/2}) \to \R, 
\quad \langle\langle s_1\otimes \mu_1, s_2\otimes\mu_2 \rangle\rangle := 
\int _M \langle s_1,
s_2\rangle \bar{\mu}_1\mu_2
\]
makes sense (whenever the integral converges).  It is easy to see that
the integral above does converge for all sections in the space
\[
L^2 (\LL\otimes |M|^{1/2}, M)\cap \Gamma(\LL\otimes |M|^{1/2}) := \left\{s\otimes \mu \in
\Gamma(\LL\otimes |M|^{1/2})\,\,\right |\left.  \int _M \langle s,s\rangle\,
|\mu|^2 < \infty\right\}.
\]
The completion of the space with respect to the Hermitian inner product gives us the ``intrinsic Hilbert space of square-integrable sections'' $L^2 (\LL\otimes |M|^{1/2}, M)$.
%\vspace{2in}

\end{document}